%% file: SIAM_article.tex
\documentclass[hidelinks,onefignum,onetabnum]{siamart220329}

\input{SIAM_shared}

\ifpdf
\hypersetup{
  pdftitle={A robust two-level overlapping preconditioner for Darcy flow in high-contrast media},
  pdfauthor={C. Ye, S. Fu, E. Chung, and J. Huang}
}
\fi

\begin{document}

\maketitle

% REQUIRED
\begin{abstract}
In this article, a two-level overlapping domain decomposition preconditioner is developed for solving linear algebraic systems obtained from simulating Darcy flow in high-contrast media. Our preconditioner starts at a mixed finite element method for discretizing the partial differential equation by Darcy's law with the no-flux boundary condition and is then followed by a velocity elimination technique to yield a linear algebraic system with only unknowns of pressure. Then, our main objective is to design a robust and efficient domain decomposition preconditioner for this system, which is accomplished by engineering a multiscale coarse space that is capable of characterizing high-contrast features of the permeability field. A generalized eigenvalue problem is solved in each non-overlapping coarse element in a communication-free manner to form the global solver, which are accompanied by local solvers originated from additive Schwarz methods but with a non-Galerkin discretization to derive the two-level preconditioner. We provide a rigorous analysis indicates that the condition number of the preconditioned system could be bounded above with several assumptions. Extensive numerical experiments with various types of three-dimensional high-contrast models are exhibited. In particular, we study the robustness against the contrast of the media as well as the influences of numbers of eigenfunctions, oversampling sizes, and subdomain partitions on the efficiency of the proposed preconditioner. Besides, strong and weak scalability performances are also examined.
\end{abstract}

% REQUIRED
\begin{keywords}
preconditioner, domain decomposition method, heterogeneous Darcy flow, multiscale space
\end{keywords}

% REQUIRED
\begin{MSCcodes}
65N55, 65F08, 65F10
\end{MSCcodes}

\section{Introduction}
Simulating subsurface fluids through porous media is important for many practical applications, such as reservoir simulation, nuclear water management, and groundwater contamination predictions. In subsurface modeling, properties of the medium such as permeability may vary by several orders of magnitude spatially and have complicated structures. Moreover, high-resolution models have become available due to advances in reservoir characterization techniques. All these factors impose prohibitively high computational burdens on existing simulators and have led to the developments of model reduction techniques such as upscaling \cite{Wu2002,Durlofsky1991,Arbogast2013}, in which people aim to homogenize high resolution models to obtain reduced ones that can be solved with permissible computational costs. Although upscaling has been successfully applied for considerable problems, it usually fails to resolve all small-scale information of the interested quantities. To overcome the bottlenecks of upscaling, various types of multiscale methods including the multiscale finite element method \cite{Hou1997} with its extensions the mixed multiscale finite element method \cite{Chen2003,Aarnes2004}, Generalized Multiscale Finite Element Methods (GMsFEM) \cite{Efendiev2013,Chung2015}, the multiscale finite volume method \cite{Lunati2006,Hajibeygi2009} and the multiscale mortar mixed finite element method \cite{Arbogast2007} have been proposed. The core idea of these methods is to solve target problems in coarse grids with carefully engineered multiscale basis functions, which are usually solutions of some well-designed sub-problems containing heterogeneity information of the original geological models. Hence, those multiscale methods can usually capture small-scale details with affordable computational resources. Nevertheless, the performance of multiscale methods may deteriorate if the permeability fields exhibit extremely high contrast and long channels \cite{Lunati2007}, which necessitates developing robust and efficient solvers for the original large-scale fine problems.

Direct solvers such as MUMPS \cite{Amestoy2001,Amestoy2019} apply a LU decomposition to the target matrix and then solve the linear system efficiently for different right-hand terms. However, in reservoir simulations, we need to update linear systems due to the change of coefficient profiles, which implies the scenario of solving multiple right-hand sides with one single matrix may not exist and the advantage of direct solvers is greatly diminished. Besides, direct solvers empirically require a huge amount of memory and perform poorly in terms of scalability. Therefore, direct solvers are not good candidates for large-scale flow simulations, which renders preconditioned iterative solvers a suitable choice. Classical two-level Domain Decomposition (DD) preconditioners have been implemented for various problems \cite{Toselli2005,Dolean2015}. In particular, the two-level restricted additive  Schwarz preconditioners combined with different solving strategies can greatly improve and expand the scope of application of the domain decomposition algorithm for large-scale flow problems \cite{Yang2019a,Yang2022,Li2020,Yang2023}. However,  the iteration numbers of standard DD preconditioners theoretically depend on the contrast ratio of the permeability field, as a result, the performance  for arbitrarily complicated media may not be guaranteed. One remedy is to replace classic polynomial-based coarse spaces with non-standard coarse spaces that embedded with heterogeneous models information \cite{Graham2007,Sarkis1997,Wang2014,Manea2016,Calvo2016,Kim2017,Dolean2012,Kim2018,Klawonn2015,Mandel2007,Galvis2010,Galvis2010a,Nataf2010,Nataf2011,Efendiev2012,Klawonn2016,Heinlein2018,Heinlein2019}. In \cite{Graham2007,Wang2014,Manea2016}, multiscale basis functions \cite{Hou1997} are utilized to construct two-level or two-grid preconditioners for highly heterogeneous incompressible flow models, it is also shown these preconditioners perform better than polynomial-based classic preconditioners. However, numerical experience suggests that these preconditioners are robust only if there are no long channels across coarse grids. To deal with more challenging models, the idea of using spectral functions was proposed and could be found in an extensive literature \cite{Calvo2016,Kim2017,Dolean2012,Kim2018,Klawonn2015,Mandel2007,Galvis2010,Galvis2010a,Nataf2010,Nataf2011,Efendiev2012,Klawonn2016,Heinlein2018,Heinlein2019}. It can now be confirmed from these works that utilizing eigenfunctions of some carefully designed local spectral problems is to some extent indispensable for constructing a robust and efficient preconditioners. However, most of these works intend to solve elliptic problems in second-order forms, few of them have been devoted to studying the elliptic problems in first-order forms, via which we can generate mass-conserving velocity fields that are valuable in long-time subsurface simulations.
Recently, a novel mixed GMsFEM \cite{Chen2020,He2021,He2021a} was proposed for solving steady state incompressible flows in the first-order form. This method contains several steps: (1) obtaining the diagonal velocity mass matrix using the trapezoidal quadrature rule; (2) eliminating the velocity field to derive a system with pressure as the only unknown; and (3) constructing multiscale basis functions for pressure field as a model reduction. However, it is observed in the numerical experiments, this method \cite{Chen2020} requires using many bases to achieve a high accuracy that is needed in solving the transport equation for two-phase flow simulations.

An efficient and robust two-grid preconditioner within the framework of mixed GMsFEM was proposed in \cite{Yang2019} for solving single-phase and two-phase incompressible heterogeneous flow problems. The robustness of this preconditioner lies in the utilization of eigenfunctions of well-designed spectral problems to form the coarse space. It is shown that the incorporation of eigenfunctions in coarse spaces can tremendously improve performance over standard two-grid preconditioners. However, the local spectral problems are defined in strongly overlapped subdomains in \cite{Yang2019}, which greatly affects the parallel efficiency.

In this paper, we develop a parallelized two-level overlapping preconditioner that is based on GMsFEM for solving the fine-scale pressure system. In each coarse element, a spectral problem is solved and eigenfunctions corresponding to those small eigenvalues are selected to form a coarse space and then a two-level preconditioner (the local solvers here are discretizations of zero-Dirichlet problems defined on overlapped subdomains) is established for the iterative solver. We emphasize that each local spectral problem is solved in a coarse element that has no overlaps with other coarse elements, which saves MPI inter-node communications and improves parallel efficiency compared with previous efforts \cite{Calvo2016,Kim2017,Dolean2012,Kim2018,Klawonn2015,Mandel2007,Galvis2010,Galvis2010a,Nataf2010,Nataf2011,Efendiev2012,Klawonn2016,Heinlein2018,Heinlein2019,Yang2019}. 
\alert{We here address preliminary results on robust DD preconditioners for the mixed form of Darcy flow problems in \cite{Efendiev2012}, where the authors also demonstrate extensions on Stokes’ and Brinkman’s equations. By introducing steam functions, they obtain a positive definite bilinear form that is suitable for constructions of eigenvalue problems. Our approach distinguishes from \cite{Efendiev2012} in that we sacrifice the pursuit of complete mixed form but focus on an alternative positive semidefinite bilinear form w.r.t.~pressure, which also allows us to obtain a velocity field that satisfies mass conservation. Additionally, our approach does not rely on partition functions, and the eigenvalue problems are formulated on a non-overlapping decomposition, resulting in a more straightforward implementation.}
Moreover, rich numerical experiments for several different types of highly heterogeneous models with a resolution up to $512^3$ are exhibited. In particular, we examine the robustness of our preconditioner against contrast ratios of media. Strong and weak scalability tests and performance comparisons with an out-of-the-box preconditioner from PETSc \cite{Balay2022a} are provided. Besides, the influence of numbers of eigenfunctions, oversampling sizes as well as subdomain partitions will be also investigated. The proposed two-level preconditioner shows an encouraging computational performance. Under several assumptions, we derive an a priori bound of condition numbers of preconditioned linear systems. Our analysis is featured for considering no-flux boundary conditions which result in singular linear systems and novel designs of local solvers.

The rest of this paper is arranged as follows. \Cref{sec:pre} introduces the basic model of subsurface flows, the definition of notations and fine-scale discretization. \Cref{sec:gms} illustrates a systematic way to construct the coarse space and the preconditioner. An analysis for the proposed preconditioner is presented in \cref{sec:anal}. Numerical results are shown in \cref{sec:num}. Finally, a conclusion is given.

\section{Preliminaries}\label{sec:pre}
\subsection{The model problem and the mixed finite element discretization}
We consider the following model problem for the unknown pressure field $p$ on a bounded Lipschitz domain $\Omega$ in $\Real^d$ where $d=2$ or $3$:
\begin{equation}\label{eq:ell1}
\left\{
\begin{aligned}
-\Div \RoundBrackets*{\kappa\nabla p}=&f \quad &&\text{in} \quad \Omega, \\
\kappa\nabla p\cdot\bm{n}=&0 \quad &&\text{on} \quad \partial\Omega,
\end{aligned}
\right.
\end{equation}
where $\partial \Omega$ is a Lipschitz continuous boundary and $\bm{n}$ is the unit outward normal vector to $\partial\Omega$. The scalar-valued coefficient $\kappa$ is the permeability of the porous medium which is highly heterogeneous in practical subsurface flow problems. By applying the divergence theorem to \cref{eq:ell1}, it is easy to validate that the source function $f$ should satisfy the so-called compatibility condition $\int_{\Omega} f\dx \bm{x}=0$. Moreover, to guarantee the uniqueness of the solution, an additional restriction $\int_{\Omega} p\dx \bm{x}=0$ should be imposed. In reservoir simulation, the flux variable
\[
\bm{v}=-\kappa\nabla p,
\]
is usually the quantity of interest. Then we can transform the second-order form \cref{eq:ell1} into a first-order form:
\begin{equation}\label{eq:orgional_equation}
\left\{
\begin{aligned}
\kappa^{-1}\bm{v}+\nabla p&=\bm{0} \quad \text{in} \quad \Omega,\quad &&\text{(Darcy's law)},\\
\Div(\bm{v})&=f \quad \text{in} \quad \Omega, \quad &&\text{(mass conservation)},\\
\bm{v}\cdot \bm{n}&=0 \quad \text{on} \quad \partial \Omega, \quad &&\text{(no-flux boundary condition)}.
\end{aligned}
\right.
\end{equation}
We choose the no-flux boundary condition because it is much more widely applied in reservoir simulation comparing with other types of boundary conditions.

We define
\[
\bm{V}\coloneqq\CurlyBrackets*{\bm{v}\in L^2(\Omega)^d\mid \Div(\bm{v})\in L^2(\Omega) \ \text{and} \  \bm{v}\cdot\bm{n}=0\;\mbox{on}\;\partial \Omega},
\]
and
\[
W\coloneqq\CurlyBrackets*{q\in L^2(\Omega)\mid \int_{\Omega}q\dx \bm{x}=0}.
\]
Then, multiplying Darcy's law by a test function $\bm{w}\in \bm{V}$ and the mass conservation equation by a $q\in W$, the weak form of \cref{eq:orgional_equation} is formulated by seeking $(\bm{v}, p)\in \bm{V}\times W$ satisfying the equations:
\begin{equation}\label{eq:weak}
\begin{aligned}
\int_\Omega\kappa^{-1}\bm{v}\cdot \bm{w} \dx \bm{x}-\int_\Omega\Div(\bm{w})p\dx \bm{x} &=0, &&\forall \bm{w}\in \bm{V},\\
-\int_\Omega\Div(\bm{v})q \dx \bm{x} &=-\int_\Omega fq\dx \bm{x},&& \forall q \in W.
\end{aligned}
\end{equation}
Let $\bm{V}_h\subset \bm{V}$ and $ W_h\subset W$ be two finite element spaces associated with a prescribed triangulation $\mathcal{T}_h$ of $\Omega$. Then the discrete problem becomes finding $(\bm{v}_h,p_h)\in \bm{V}_h\times W_h$ such that
\begin{equation}\label{eq:weak2}
\begin{aligned}
\int_\Omega\kappa^{-1}\bm{v}_h\cdot \bm{w}_h\dx \bm{x}-\int_\Omega \Div(\bm{w}_h)p_h\dx \bm{x}&=0, &&\forall \bm{w}_h\in \bm{V}_h,\\
-\int_\Omega\Div(\bm{v}_h)q_h \dx \bm{x}&=-\int_\Omega fq_h \dx \bm{x}, &&\forall q_h \in W_h.
\end{aligned}
\end{equation}
There are a number of well-known families of mixed finite element spaces \cite{Auricchio2017} of $\bm{V}_h\times W_h$, here we use the lowest-order Raviart-Thomas finite element spaces (cf. \cite{Boffi2013}) which is denoted by $\text{RT}_0$ in this paper. We express the solution $(\bm{v}_h, p_h)$ as
\[
\bm{v}_h=\sum_{e\in\mathcal{E}^0_h}v_e\bm{\phi}_e \quad \text{and} \quad p_h=\sum_{\tau\in\mathcal{T}_h}p_\tau q_\tau,
\]
where $\mathcal{E}^0_h$ is the set of internal edges/faces, $\CurlyBrackets*{\bm{\phi}_e}_{e\in \mathcal{E}^0_h}$ and $\CurlyBrackets*{q_\tau}_{\tau\in \mathcal{T}_h}$ are the bases\footnote{Generally, those bases should depend on $h$, and we drop it here for simplicity.} of $\bm{V}_h$ and $W_h$. Then algebraic representation of \cref{eq:weak2} is
\begin{equation}\label{eq:fine_system}
\begin{pmatrix}
\mathsf{M} & -\mathsf{B}^\intercal \\
-\mathsf{B} & \mathsf{0}\\
\end{pmatrix}
\begin{pmatrix}
\mathsf{v} \\ \mathsf{p}
\end{pmatrix}=
\begin{pmatrix}
0\\-\mathsf{f}
\end{pmatrix},
\end{equation}
where $\mathsf{M}$ is a symmetric, positive definite matrix, $\mathsf{v}$, $\mathsf{p}$ and $\mathsf{f}$ are column vectors.

\subsection{The velocity elimination technique}\label{subsec:vet}
Designing an efficient preconditioner for solving the saddle point system \cref{eq:fine_system} is not an easy task. Fortunately, it is shown in \cite{Arbogast1997} that if the trapezoidal quadrature rule is applied to compute the integration $\int_{\Omega}\kappa^{-1}\bm{v}_h\cdot\bm{w}_h\dx \bm{x}$ on each element rather than the accurate rule, then we could obtain a diagonal velocity mass matrix $\tilde{\mathsf{M}}$ that can be inverted straightforwardly. Therefore, the unknown $\mathsf{v}$ in \cref{eq:weak2} can be eliminated and \cref{eq:fine_system} boils down to
\begin{equation}\label{eq:p}
\mathsf{A}\mathsf{p}=\mathsf{f},
\end{equation}
where $\mathsf{A}\coloneqq\mathsf{B}\tilde{\mathsf{M}}^{-1}\mathsf{B}^\intercal$ is a symmetric and semi-positive definite matrix. We call this technique as velocity elimination, the equivalence between velocity elimination and a cell-centered finite difference scheme in rectangular meshes can be referred in \cite{Arbogast1997}.

Since the velocity elimination technique is closely related with the construction of our preconditioner, we will restate it in detail here. For simplicity, we consider a rectangular $(d=2)$ domain $\Omega=(0, L_x)\times(0, L_y)$\footnote{In previous equations $\bm{x}$ is the vector-valued spatial variable (such as \cref{eq:weak,eq:weak2}), while here we also use subscripts $x$ and $y$ for $x$- and $y$-direction.} with a uniform structured mesh $M_x\times M_y$, which says that each element is of the size of $h_x\times h_y$ with $h_x=L_x/M_x$ and $h_y=L_y/M_y$. Note that this mesh should be at the finest resolution and $\kappa$ should be element-wisely constant. For an element $\tau=(i, i+1)h_x\times(j, j+1)h_y$ with $0\leq i < M_x$ and $0 \leq j < M_y$, the trapezoidal quadrature rule gives
\[
\begin{aligned}
&\int_{\tau} \kappa^{-1} \bm{v}_h\cdot\bm{w}_h \dx \bm{x} \\
\approx &\frac{\RoundBrackets*{\kappa|_\tau}^{-1}}{4}\CurlyBrackets*{(\bm{v}_h\cdot\bm{w}_h)(p_1)+(\bm{v}_h\cdot\bm{w}_h)(p_2)+(\bm{v}_h\cdot\bm{w}_h)(p_3)+(\bm{v}_h\cdot\bm{w}_h)(p_4)} h_xh_y,
\end{aligned}
\]
where $\kappa|_\tau$ is the constant value of $\kappa$ in $\tau$, $p_1,\dots,p_4$ are the four corner points of $\tau$. According to the definition of $\bm{\phi}_e$ (cf. \cite{Boffi2013}), for $p_1=(ih_x, jh_y)$, we have $(\bm{v}_h\cdot\bm{w}_h)(p_1)=v_{e_1}w_{e_1}+v_{e_2}w_{e_2}$, where $e_1$ and $e_2$ are the edges $\CurlyBrackets{ih_x}\times (j, j+1)h_y$ and $(i, i+1)h_x\times \CurlyBrackets{jh_y}$ respectively. We can handle $(\bm{v}_h\cdot\bm{w}_h)(p_2)$, $(\bm{v}_h\cdot\bm{w}_h)(p_3)$ and $(\bm{v}_h\cdot\bm{w}_h)(p_4)$ similarly and then obtain an approximation of $\int_\Omega\kappa^{-1}\bm{v}_h\cdot \bm{w}_h\dx \bm{x}$ as $h_x h_y\sum_{e\in \mathcal{E}_h^0} \kappa_e^{-1} v_ew_e$ with $\kappa_e$ is the \emph{harmonic average} of $\kappa_{e,+}$ and $\kappa_{e,-}$, where $\kappa_{e,+}$ and $\kappa_{e,-}$ are two values of $\kappa$ in the two adjacent elements along the internal edge $e$. Note that $\int_\Omega\kappa^{-1}\bm{v}_h\cdot \bm{w}_h\dx \bm{x} \approx h_x h_y\sum_{e\in \mathcal{E}_h^0} \kappa_e^{-1} v_ew_e$ actually provides us the expression of $\tilde{\mathsf{M}}$. For $\int_\Omega \Div(\bm{w}_h)p_h\dx \bm{x}$, recalling that $p_h$ is also element-wisely constant and applying the divergence theorem for each element, we have
\[
\begin{aligned}
\int_\Omega \Div(\bm{w}_h)p_h\dx \bm{x}&=\sum_{\tau\in\mathcal{T}_h}\int_\tau \Div(\bm{w}_h)p_h\dx \bm{x} \\
&= \sum_{\tau\in\mathcal{T}_h} p_h|_\tau \int_\tau \Div(\bm{w}_h)\dx \bm{x}=\sum_{\tau\in\mathcal{T}_h} p_h|_\tau\int_{\partial \tau} \bm{w}_h\cdot \bm{n}\dx \sigma \\
&=-\sum_{e\in\mathcal{E}_h^0} \abs{e}\DSquareBrackets{p_h}_ew_e,
\end{aligned}
\]
where $\abs{e}$ is the length of $e$ and $\DSquareBrackets{p_h}_e$ is the jump of $p_h$ across $e$ towards the positive $x$- or $y$-direction. We hence cancel out $w_e$ via $h_x h_y\sum_{e\in \mathcal{E}_h^0} \kappa_e^{-1} v_ew_e+\sum_{e\in\mathcal{E}_h^0} \abs{e}\DSquareBrackets{p_h}_ew_e=0$ and obtain $v_e=-\kappa_e\abs{e}\DSquareBrackets{p_h}_e/(h_xh_y)$. Moreover, combining $\int_\Omega \Div(\bm{v}_h)q_h\dx \bm{x}=-\sum_{e\in\mathcal{E}_h^0} \abs{e}\DSquareBrackets{q_h}_ev_e$, we could derive a variational form expression of \cref{eq:p} as
\begin{equation}\label{eq:varia vet}
\sum_{e\in\mathcal{E}_h^0} \kappa_e \DSquareBrackets{p_h}_e\DSquareBrackets{q_h}_e\frac{\abs{e}^2}{h_xh_y} = \int_\Omega f q_h \dx \bm{x}, \quad \forall q_h \in W_h.
\end{equation}
\alert{For $d=3$, the summation $\sum_{e\in\mathcal{E}_h^0}$ in \cref{eq:varia vet} should be interpreted as the sum over all internal element faces, and $\DSquareBrackets{\cdot}_e$ represents the jump across the face $e$. In contrast to the no-flux boundary condition, for zero-Dirichlet boundary value problems, the velocity across the boundary is not known priorly. Therefore, following the velocity technique, additional terms will be included in \cref{eq:varia vet}. Although the precise formulation is omitted here for the sake of simplicity.}
% The extensions to $d=3$ and zero-Dirichlet boundary value problems are similar.

Once the pressure unknowns $\mathsf{p}$ are available by solving \cref{eq:p}, the velocity unknowns $\mathsf{v}$ can be recovered via $\mathsf{v}=\tilde{\mathsf{M}}^{-1}\mathsf{B}^\intercal \mathsf{p}$. We emphasize that the velocity field $\bm{v}_h$ computed here is element-wisely mass conservative, which is crucial for long-time subsurface flow modeling.

\subsection{Two-level overlapping preconditioners}
In this subsection, the idea of two-level overlapping preconditioners (ref. \cite{Toselli2005}) for solving \cref{eq:p} is briefly presented. Assume the domain $\Omega$ is divided into a union of \emph{disjoint} polygonal subdomains denoted by $\CurlyBrackets*{K_i}_{i=1}^N$ with each $K_i$ consisting of fine elements from $\mathcal{T}_h$. Therefore, we obtain a coarse ``triangulation'' of $\Omega$ that is denoted by $\mathcal{T}_H$, and we will also call a subdomain $K_i$ as a coarse element. Let $\CurlyBrackets*{K_i^m}_{i=1}^N$ be an overlapping partition by extending $m$ layers of fine elements to each coarse element $K_i$. In later presentation, we also name $K_i^m$ as an oversampling coarse element, and the term ``oversampling'' is coined in the oversampling multiscale finite element method \cite{Hou1997}. \Cref{fig:grid} is an illustration of the two-scale mesh, a fine element $\tau$, a coarse element $K_i$ and its oversampling coarse element $K_i^m$ with $m=2$.

\begin{figure}[tbhp]
\centering
\begin{tikzpicture}[scale=1.5]
\draw[step=0.25, gray, thin] (-0.4, -0.4) grid (4.4, 4.4);
\draw[step=1.0, black, very thick] (-0.4, -0.4) grid (4.4, 4.4);
\foreach \x in {0,...,4}
\foreach \y in {0,...,4}{
\fill (1.0 * \x, 1.0 * \y) circle (1.5pt);
}
\fill[brown, opacity=0.4] (1.0, 1.0) rectangle (2.0, 2.0);
\node at (1.5, 1.5) {$K_i$};
\draw [dashed, very thick, fill=gray, opacity=0.6] (0.5, 0.5) rectangle (2.5, 2.5);
\node[above right] at (2.0, 2.0) {$K_i^2$};
\draw [dashed, very thick, fill=cyan, opacity=0.5] (3.25, 1.25) rectangle (3.5, 1.5);
\node at (3.375, 1.375) {$\tau$};
\end{tikzpicture}
\caption{An illustration of the two-scale mesh, a fine element $\tau$, a coarse element $K_i$ and its oversampling coarse element $K_i^m$ with $m=2$.}
\label{fig:grid}
\end{figure}
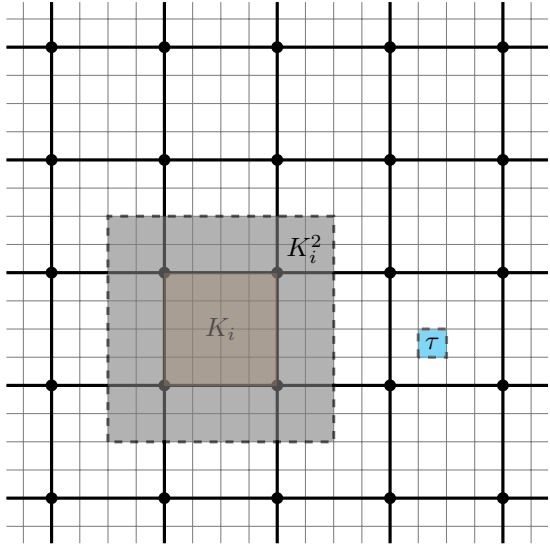

We introduce coarse basis functions $\CurlyBrackets*{\Phi_i}_{i=1}^{N^\text{c}}$ associated with the coarse mesh $\mathcal{T}_H$, where $N^\text{c}$ is the number of bases, and $N^\text{c}$ in generally is different from the number of coarse elements $N$. Note that each basis should be resolved on the fine mesh $\mathcal{T}_h$, and we can hence represent the coarse space as
\[
W_H^\text{c} \coloneqq \Span\CurlyBrackets*{\Phi_i}_{i=1}^{N^\text{c}}
\]
and its projection to the fine space as a prolongation matrix $\mathsf{R}_0^\intercal$ with each column of $\mathsf{R}_0^\intercal$ corresponding to a coarse basis. The coarse space can simply be $\text{RT}_0$ space for pressure on the coarse mesh $\mathcal{T}_H$, it can also be a \emph{generalized multiscale finite element space}. After determining $W_H^\text{c}$, the matrix of the global solver can be obtained by $\mathsf{A}_0\coloneqq \mathsf{R}_0 \mathsf{A}\mathsf{R}_0^\intercal$, where the size of $\mathsf{A}_0$ is $N^\text{c}\times N^\text{c}$. We emphasize that constructing a low-dimensional but effective coarse space is pivotal in designing two-level preconditioners.

An indispensable component in two-level preconditioners is so-called local solvers. In our paper, the local solvers are built from discretizing zero-Dirichlet problems defined on \emph{oversampling} coarse elements $\CurlyBrackets*{K_i^m}_{i=1}^N$. That is, we solve the following equation on the fine mesh $\mathcal{T}_h$ via the $\text{RT}_0$ method:
\begin{equation}\label{eq:local}
\left\{
\begin{aligned}
\kappa^{-1}\bm{v}+\nabla p &= \bm{0} \quad &&\text{in} \quad K_i^m,\\
\nabla\cdot \bm{v} &= f \quad &&\text{in} \quad K_i^m, \\
p &= 0 \quad && \text{on} \quad \partial K_i^m \cap \Omega, \\
\bm{v}\cdot \bm{n} &= 0 \quad && \text{on} \quad \partial K_i^m \cap \partial\Omega.
\end{aligned}
\right.
\end{equation}
Note that we utilize the velocity elimination technique again to obtain a linear system that is only related to pressure. We denote $\mathsf{A}_i$ the matrix of the local solver corresponding to $K_i^m$, and $\mathsf{R}_i$ the restriction matrix that restricts the Degrees of Freedom (DoF) of pressure from $\Omega$ to $K_i^m$. Note that $\mathsf{A}_i$ is not equal to $\mathsf{R}_i \mathsf{A} \mathsf{R}_i^\intercal$, and this fact is different from the discretization by the Lagrange finite element method. Except for special situations (e.g., $\partial K_i^m \cap \Omega = \varnothing$), the matrix $\mathsf{A}_i$ is invertible, and we will always assume that $\mathsf{A}_i^{-1}$ exists.

The fine-scale linear system \cref{eq:p} will be solved by preconditioned iterative solvers such as Generalized Minimal Residual (GMRES) methods with a two-level overlapping preconditioner of the form:
\[
\mathsf{P}^{-1}=\mathsf{R}_0^\intercal \mathsf{A}_0^{\dagger}\mathsf{R}_0+\sum_{i=1}^{N}\mathsf{R}_i^\intercal \mathsf{A}_i^{-1}\mathsf{R}_i,
\]
where $\mathsf{A}_0^{\dagger}$ is the pseudoinverse of $\mathsf{A}_0$ because $\mathsf{A}_0$ is generally singular and ``$-1$'' of $\mathsf{P}^{-1}$ is just a conventional notation. The definition of this preconditioner indicates that it needs to solve a global system and multiple local problems in each iteration. In practice, we will factorize $\mathsf{A}_0$ and $\mathsf{A}_i$ during the preparation phase, which saves computing time on solving precondition systems in the latter iteration phase.

\section{GMsFEM based two-level overlapping preconditioner}\label{sec:gms}
The choice of the coarse space $W_H^\text{c}$ has a significant influence on the performance, and we will adopt the methodology of GMsFEMs to construct a coarse space that incorporates heterogeneity information of $\kappa$. Follow the simplification in \cref{subsec:vet}, we will still consider a rectangular domain with a uniform structured mesh, and the generalization to unstructured meshes is also possible (see \cite{He2021}). Meanwhile, the coarse mesh $\mathcal{T}_H$ is also structured and each coarse element $K_i$ is a rectangular subdomain.

In GMsFEMs \cite{Efendiev2013}, two bilinear forms $a_i(\cdot, \cdot)$ and $s_i(\cdot, \cdot)$ corresponding to the $i$-th subdomain\footnote{Subdomains $\CurlyBrackets*{\omega_i}$ are overlapped in \cite{Efendiev2013}, while coarse elements in our paper are all disjointed.} are important in selecting effective low-dimensional representations from snapshot spaces. According to \cref{eq:varia vet}, we define $a_i(\cdot, \cdot)$ as
\begin{equation}\label{eq:a_i_func}
a_i(q_h', q_h)\coloneqq \sum_{e\in \mathcal{E}_h^0(K_i)}\kappa_e\DSquareBrackets{q_h'}_e\DSquareBrackets{q_h}_e\frac{\abs{e}^2}{h_xh_y}, \quad \forall q_h,q_h'\in W_h(K_i),
\end{equation}
where $\mathcal{E}_h^0(K_i)$ is the internal edge set in the coarse element $K_i$ and $W_h(K_i)$ is the restriction space of $W_h$ on $K_i$. Due to that our discretization is not a Galerkin method, the expression of $a_i(\cdot,\cdot)$ is different from the original one in \cite{Efendiev2013}. We define $s_i(\cdot,\cdot)$ as
\begin{equation}\label{eq:s_i_func}
s_i(q_h',q_h) \coloneqq \int_{K_i}\tilde{\kappa} q_h' q_h \dx \bm{x}, \quad \forall q_h,q_h'\in W_h(K_i),
\end{equation}
%where $\tilde{\kappa}$ is a modified permeability field which could simply be $\kappa$ or derived from theoretical analysis in \cref{sec:anal}. 
\alert{where $\tilde{\kappa}$ must satisfy \textit{Assumption B} as described in \cref{sec:anal}. In addition, we explain in \cref{sec:num} that it is also reasonable to set $\tilde{\kappa}=\kappa$.}
Then, we solve the following spectral problem in each coarse element $K_i$:
\begin{equation}\label{eq:spepb}
\sum_{e\in \mathcal{E}_h^0(K_i)}\kappa_e\DSquareBrackets{\Phi_h}_e\DSquareBrackets{q_h}_e\frac{\abs{e}^2}{h_xh_y} = \lambda \int_{K_i}\tilde{\kappa} \Phi_h q_h \dx \bm{x}, \quad \forall q_h \in W_h(K_i),
\end{equation}
where $\Phi_h$ is an eigenvector corresponding to the eigenvalue $\lambda$. Note that \cref{eq:spepb} is the discretization of the spectral problem---find $\Phi \in H^1(K_i)$ such that $-\Div\RoundBrackets*{\kappa \nabla \Phi}=\lambda\tilde{\kappa} \Phi$ in $K_i$---via $\text{RT}_0$ elements and the velocity elimination technique. These eigenvectors supported in coarse elements are solutions of local spectral problems that contain fine-scale information $\kappa$, therefore we also call them multiscale basis functions. After solving the eigenvalue problem \cref{eq:spepb}, eigenvectors $\CurlyBrackets{\Phi_{h,i}^j}_{j=1}^{L_i}$ corresponding to the $L_i$ smallest eigenvalues will be collected to form the local coarse space $W_H^\text{c}(K_i)$ as $W_H^\text{c}(K_i)=\Span\CurlyBrackets{\Phi_{h,i}^j:j=1,\dots,L_i}$. Then, the global coarse space $W_H^c$ is the direct sum\footnote{Two different local coarse space are associated to two different subdomains, which says the direct sum operation is inappropriate. Strictly, the direct sum here should implicitly come after performing zero-extensions of bases to the whole domain.} of $W_H^c(K_i)$ as $W_H^\text{c}=\bigoplus_{i=1}^N W_H^\text{c}(K_i)$. The dimension of $W_H^c$ equals the total number of eigenvectors, that is $N^\text{c}=\sum_{i=1}^{N} L_i$.  We emphasize again that \cref{eq:spepb} is solved in each coarse element $K_i$ that has no overlap with its neighbors, and therefore no communication is required when adopting parallel computing. Comparing with previous efforts \cite{Galvis2010,Galvis2010a,Nataf2010,Nataf2011}, in which local problems are solved in overlapping subdomains, the advantage of our method in parallel computing implementation is remarkable.

\begin{remark}
\alert{
In the original paper on GMsFEMs \cite{Efendiev2013}, the authors propose the use of $\kappa$-harmonic spaces as snapshot spaces, which involves solving the spectral problem in \cref{eq:spepb} on a subspace of $W_h(K_i)$. In our paper, we advocate for the approach in \cref{eq:spepb} for eigenvalue problems, which could be understood by utilizing full $W_h(K_i)$ as snapshot spaces. While using $\kappa$-harmonic spaces can reduce the degrees of freedom in spectral problems, it can also result in dense eigenvalue systems. Furthermore, implementing the use of $\kappa$-harmonic spaces is more complex, making it challenging to assess the actual performance improvement. Additionally, our numerical findings indicate that in parallel computing scenarios, the time spent on spectral problems does not significantly contribute to the overall computational time.
}
\end{remark}
The bilinear form $s_i(\cdot,\cdot)$ is symmetric and positive definite, while $a_i(\cdot,\cdot)$ is symmetric but always semi-definite, which implies the first eigenvalue of \cref{eq:spepb} is always $0$. Moreover, it is easy to see there is a unique and normalized eigenvector corresponding to the eigenvalue $0$, and this eigenvector is exactly the constant vector. Note that the $\text{RT}_0$ spaces $\mathbf{V}_H\times W_H$ on the coarse mesh $\mathcal{T}_H$, bases of the pressure part $W_H$ are also piece-wisely constant. Then, our coarse space can be viewed as an enrichment of $W_H$, which is a standard coarse space in two-grid preconditioners \cite{Briggs2000}. Recall that the no-flux boundary condition is imposed for our model problem \cref{eq:ell1}, and the matrix $\mathsf{A}$ of the fine-scale linear system \cref{eq:p} is actually singular with constant vectors as its kernel. Note that the coarse space $W_H^\text{c}$ also contains the kernel of $\mathsf{A}$, and this implies that $W_H^\text{c}$ inherits some physical information of the original system and can be interpreted as an appropriate upscaling of the fine-scale space.

\Cref{fig:basis} demonstrates eigenvalues computed by setting different permeability profiles. In the left column, two types of media (3- and 5-channel configurations) that contain long channels are visualized, and we will fix $\kappa\equiv 1$ in the background region (purple) and $\kappa \equiv \kappa^*$ in channels (yellow), where $\kappa^*$ is from $\CurlyBrackets*{10^0,10^1,10^2,10^3,10^4}$ to see the effect of contrast ratios. We plot the $11$ smallest eigenvalues of \cref{eq:spepb} with $\tilde{\kappa}=\kappa$ in the middle column, with $\tilde{\kappa}=1$ in the right column, where plots in the first row are corresponding to the 3-channel configuration and the second row is for the 5-channel configuration accordingly. We can observe the patterns of eigenvalues for $\tilde{\kappa}=\kappa$ and $1$ are significantly different. Checking vertical axes of the plots in the right and middle columns, eigenvalues corresponding to $\tilde{\kappa}=1$ seem to be sensitive to different configurations, while eigenvalues to $\tilde{\kappa}=\kappa$ are more stable. An interesting phenomenon could be noticed in the middle column is for the 3-channel (5-channel) configuration, the $3$ ($5$) smallest eigenvalues will approach to $0$ as increasing contrast ratios. Actually, it is shown mathematically that in \cite{Galvis2010}, if there are $m$ channels, the $m$ smallest eigenvalues may approach to $0$ as increasing contrast ratios, which is not favored in building robust preconditions. Generally, We need to include enough eigenvectors into our coarse space, and the number of channels may provide a guidance. However, for a complex medium, it is almost impossible to have a priori knowledge about the precise number of channels. In practice, we may preset a number or choose eigenvectors adaptively such that the largest eigenvalue crosses a certain threshold.

\begin{figure}[tbhp]
\centering
\includegraphics[width=\textwidth]{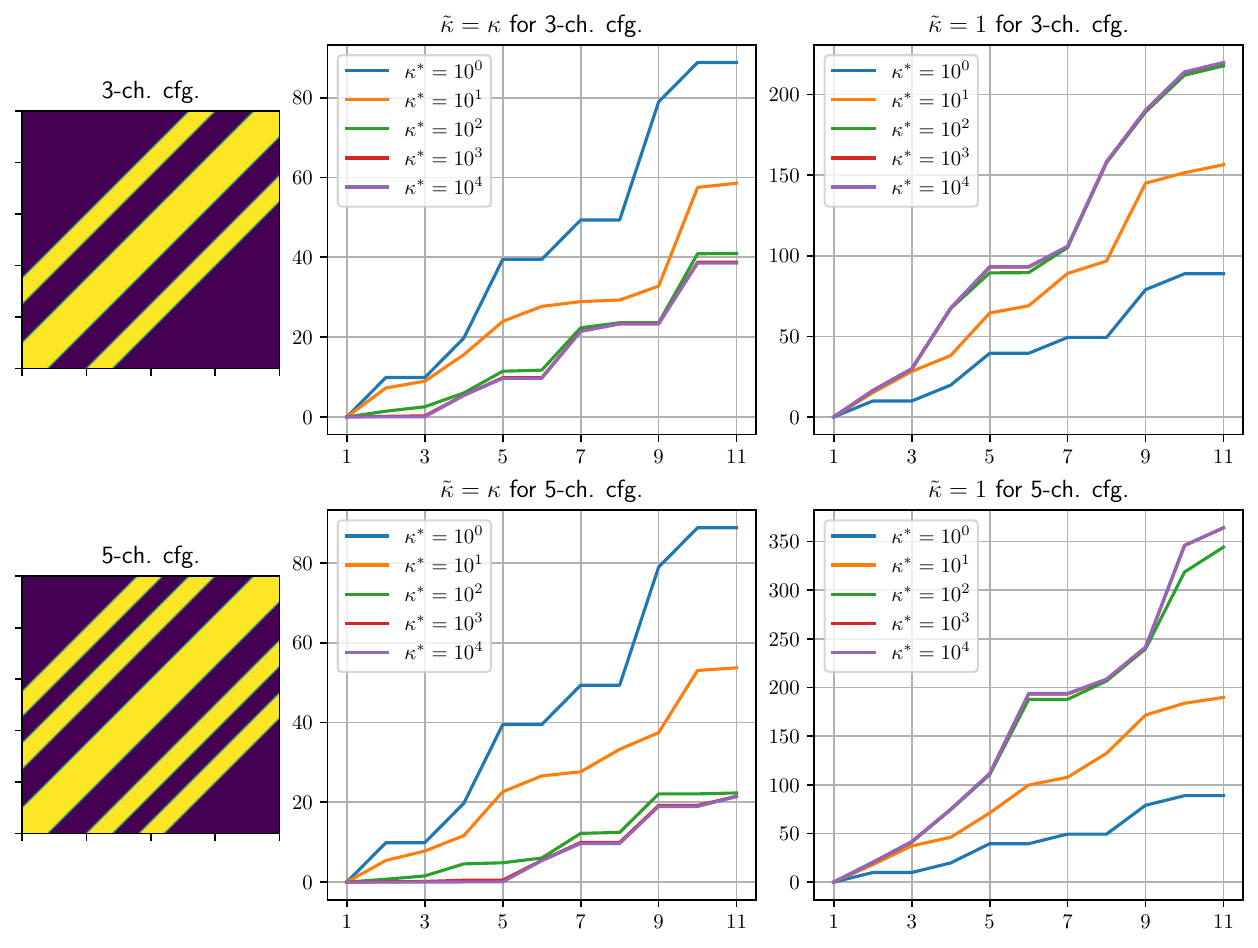}
\caption{(\textbf{left column}) the 3- and 5-channel configurations; (\textbf{middle column}) the $11$ smallest eigenvalues calculated by setting $\tilde{\kappa}=\kappa$ w.r.t. different $\kappa^*$ that is the value of $\kappa$ in channels, and the top (bottom) plot is corresponding to the 3-channel (5-channel) configuration; (\textbf{right column}) the $11$ smallest eigenvalues calculated by setting $\tilde{\kappa}=1$ w.r.t. different $\kappa^*$, and the top (bottom) plot is corresponding to the 3-channel (5-channel) configuration.}
\label{fig:basis}
\end{figure}

We summarize our method of constructing the preconditioner in \cref{alg:preconditioner}. Basically, the method consists of two phases---preparation and iteration. Note that every step listed in \cref{alg:preconditioner} could be accelerated in a parallel framework. Although the exact iteration phase depends on the iterative solver, we just focus on applying $\mathsf{P}^{-1}\mathsf{r}$, where $\mathsf{r}$ is usually a residual vector.

\begin{algorithm}
\caption{GMsFEM based two-level overlapping preconditioner}
\label{alg:preconditioner}
\begin{algorithmic}[1]
\STATE{\textbf{Preparation phase}:}
\INDSTATE{Set up the decomposition $\mathcal{T}_H$ for the domain $\Omega$; load the coefficient profile $\kappa$; set the source term $f$.}
\INDSTATE{Construct the matrix $\mathsf{A}$ and the right-hand vector $\mathsf{f}$ of \cref{eq:p}.}
\INDSTATE{For each oversampling coarse element $K_i^m$, construct the matrix $\mathsf{A}_i$ by discretizing \cref{eq:local} via velocity elimination; factorize $\mathsf{A}_i$.}
\INDSTATE{For each coarse element $K_i$, solve the spectral problem \cref{eq:spepb} and obtain $L_i$ eigenvectors; form the coarse space $W_H^\text{c}$ or the prolong matrix $\mathsf{R}_0^\intercal$.}
\INDSTATE{Construct the matrix $\mathsf{A}_0$ via the coarse space $W_H^\text{c}$; factorize $\mathsf{A}_0$.}

\STATE{\textbf{Iteration phase} (apply $\mathsf{P}^{-1}\mathsf{r}$):}
\INDSTATE{Restrict $\mathsf{r}$ to the oversampling coarse element $K_i^m$ and obtain $\mathsf{r}_i$; use factorized $\mathsf{A}_i$ to get $\mathsf{A}_i^{-1}\mathsf{r}_i$; sum all $\mathsf{A}_i^{-1}\mathsf{r}_i$ according to DoF in the fine mesh.}
\INDSTATE{Project $\mathsf{r}$ to the coarse space $W_H^\textsc{c}$ and obtain $\mathsf{r}_0=\mathsf{R}_0\mathsf{r}$; use factorized $\mathsf{A}_0$ to get $\mathsf{A}_0^{\dagger}\mathsf{r}_0$; project the solution $\mathsf{A}_0^{\dagger}\mathsf{r}_0$ back to the fine space with the coarse bases.}
\INDSTATE{Take a summation of the resulting vectors in the aforementioned two steps and obtain $\mathsf{P}^{-1}\mathsf{r}$.}
\end{algorithmic}
\end{algorithm}

\section{Analysis}\label{sec:anal}
For iterative solvers, a smaller iteration number usually means a better performance\footnote{Actually, the performance or the elapsed time for a preconditioned iterative solver cannot be solely determined by numbers of iterations. To see this, just assume the precondition system is the original system, then only one iteration is needed while there is no acceleration benefiting from preconditioning at all.}, and numbers of iterations are dramatically affected by the condition number of the preconditioned matrix \cite{Saad2003}. However, the canonical definition of condition numbers is for invertible matrices, while the matrix $\mathsf{A}$ in \cref{eq:p} is singular. According to \cite{Brown1997}, for a linear system $\mathsf{M}\mathsf{x}=\mathsf{b}$, if the right-hand vector $\mathsf{b}$ belongs to $\Image(\mathsf{M})$ the image of $\mathsf{M}$, then GMRES can safely find a least-squares solution or the pseudoinverse solution. Moreover, the paper \cite{Brown1997} suggests a redefinition of the condition number for $\mathsf{A}$ as
\[
\Cond(\mathsf{A}) = \frac{\sigma_{\max}(\mathsf{A})}{\sigma_2(\mathsf{A})},
\]
where $\sigma_{\max}(\mathsf{A})$ is the largest singular value of $\mathsf{A}$ and $\sigma_2(\mathsf{A})$ is the second least singular value. Hence, this section aims to show $\Cond(\mathsf{P}^{-1}\mathsf{A})$ can be bounded above with suitable assumptions.

The main tool is the so-called fictitious space lemma (ref. \cite{Griebel1995}). Before presenting this lemma, we introduce several notations. Let $n$ be the number of total DoF of pressure, $n_i$ be the number of DoF of pressure restricted on the coarse element $K_i$, and $n_i^m$ be the number of DoF of pressure restricted on the oversampling coarse element $K_i^m$, $\underline{V}$ be the product space $\Real^{N^\text{c}}\times \Real^{n_1^m} \times \dots \times \Real^{n_N^m}$. Then we have $\mathsf{A} \in \Real^{n\times n}$, and $\mathsf{A}_i \in \Real^{n_i^m\times n_i^m}, \mathsf{R}_i \in \Real^{n_i^m \times n}$ for $i=1,\dots,N$. We define $\mathsf{a}(\cdot, \cdot)$ as a symmetric, positive semi-definite bilinear form on $\Real^n$ by
\[
\mathsf{a}(\mathsf{v}, \mathsf{w}) = \mathsf{v} \cdot \mathsf{A} \mathsf{w}, \quad \forall \mathsf{v}, \mathsf{w} \in \Real^n,
\]
which is exactly the left-hand section of \cref{eq:varia vet} in a linear algebraic form. Recall that the dimension of the coarse space $W_H^\text{c}$ is $N^\text{c}$, which gives $\mathsf{A}_0\in \Real^{N^\text{c} \times N^\text{c}}$ and $\mathsf{R}_0 \in \Real^{N^\text{c}\times n}$. The fictitious space lemma is stated below, and we avoid using abstract Hilbert spaces here.

\begin{lemma} \label{lem:fsl}
Let $\mathsf{b}(\cdot,\cdot)$ be a symmetric, positive semi-definite bilinear form on $\underline{V}$ with a matrix representation as
\[
\mathsf{b}(\underline{\mathsf{v}}, \underline{\mathsf{w}}) = \underline{\mathsf{v}} \cdot \underline{\mathsf{B}} \,\underline{\mathsf{w}}, \quad \forall \underline{\mathsf{v}}, \underline{\mathsf{w}} \in \underline{V},
\]
$\mathcal{R}:\underline{V} \rightarrow \Real^n$ be a linear map, and $\mathcal{R}^*:\Real^n\rightarrow \underline{V}$ be an adjoint operator by
\[
\mathcal{R}\underline{\mathsf{v}} \cdot \mathsf{w} = \underline{\mathsf{v}} \cdot \mathcal{R}^* \mathsf{w}, \quad \forall \underline{\mathsf{v}} \in \underline{V}, \mathsf{w} \in \Real^n.
\]
Assume that $\mathcal{R}$ maps $\ker(\underline{\mathsf{B}})$ into $\ker(\mathsf{A})$, and there are positive constants $C_T$ and $C_R$, such that for all $\mathsf{u} \in \Real^n$ there exists a $\underline{\mathsf{v}} \in \underline{V}$ with $\mathsf{u}=\mathcal{R}\underline{\mathsf{v}}$ and
\begin{equation}\label{eq:fsl1}
C_T\, \mathsf{b}(\underline{\mathsf{v}}, \underline{\mathsf{v}}) \leq \mathsf{a}(\mathsf{u}, \mathsf{u}),
\end{equation}
and also
\begin{equation}\label{eq:fsl2}
\mathsf{a}(\mathcal{R}\underline{\mathsf{w}}, \mathcal{R}\underline{\mathsf{w}}) \leq C_R\, \mathsf{b}(\underline{\mathsf{w}}, \underline{\mathsf{w}}),\quad \forall \underline{\mathsf{w}} \in \underline{V}.
\end{equation}
Then for all $\mathsf{u} \in \Real^n$ with $\mathsf{a}(\mathsf{u}, \mathsf{u}) > 0$,
\[
C_T\, \mathsf{a}(\mathsf{u}, \mathsf{u}) \leq \mathsf{a}(\mathcal{R}\underline{\mathsf{B}}^\dagger\mathcal{R}^*\mathsf{A}\mathsf{u}, \mathsf{u}) \leq C_R\, \mathsf{a}(\mathsf{u}, \mathsf{u}),
\]
where $\underline{\mathsf{B}}^\dagger$ is the pseudoinverse of $\underline{\mathsf{B}}$.
\end{lemma}
The proof of \cref{lem:fsl} is an extension of the proof in \cite{Griebel1995}, while several modifications required due to that $\mathsf{a}(\cdot, \cdot)$ and $\mathsf{b}(\cdot,\cdot)$ are no longer positive definite.
\begin{proof}
Note that $\underline{\mathsf{v}}-\underline{\mathsf{B}}^\dagger \underline{\mathsf{B}}\,\underline{\mathsf{v}}\in \ker(\underline{\mathsf{B}})$ for all $\underline{\mathsf{v}}\in \underline{V}$. Thanks to the assumption that $\mathcal{R}$ maps $\ker(\underline{\mathsf{B}})$ into $\ker(\mathsf{A})$, we can show that
\[
\mathcal{R}\underline{\mathsf{v}}\cdot\mathsf{A}\mathsf{w}=\mathcal{R}\underline{\mathsf{B}}^\dagger \underline{\mathsf{B}}\,\underline{\mathsf{v}}\cdot\mathsf{A}\mathsf{w}
\]
for all $\mathsf{w}\in \Real^n$. Then, for all $\mathsf{u} \in \Real^n$ with $\mathsf{a}(\mathsf{u}, \mathsf{u}) > 0$, choosing $\underline{\mathsf{v}} \in \underline{V}$ such that $\mathsf{u}=\mathcal{R}\underline{\mathsf{v}}$, we have
\[
\begin{aligned}
\mathsf{a}(\mathsf{u}, \mathsf{u})= &\mathsf{u}\cdot\mathsf{A}\mathsf{u}=\mathcal{R}\underline{\mathsf{v}}\cdot\mathsf{A}\mathsf{u}=\mathcal{R}\underline{\mathsf{B}}^\dagger \underline{\mathsf{B}}\,\underline{\mathsf{v}}\cdot\mathsf{A}\mathsf{w}=\underline{\mathsf{B}}^\dagger \underline{\mathsf{B}}\,\underline{\mathsf{v}}\cdot \mathcal{R}^*\mathsf{A}\mathsf{u} = \sqrt{\underline{\mathsf{B}}} \, \underline{\mathsf{v}} \cdot \sqrt{\underline{\mathsf{B}}}\, \underline{\mathsf{B}}^\dagger \mathcal{R}^*\mathsf{A}\mathsf{u} \\
\leq & \RoundBrackets*{\sqrt{\underline{\mathsf{B}}} \, \underline{\mathsf{v}} \cdot \sqrt{\underline{\mathsf{B}}} \, \underline{\mathsf{v}}}^{1/2} \RoundBrackets*{\sqrt{\underline{\mathsf{B}}}\, \underline{\mathsf{B}}^\dagger \mathcal{R}^*\mathsf{A}\mathsf{u}\cdot\sqrt{\underline{\mathsf{B}}}\, \underline{\mathsf{B}}^\dagger \mathcal{R}^*\mathsf{A}\mathsf{u}}^{1/2} \\
= & \mathsf{b}(\underline{\mathsf{v}}, \underline{\mathsf{v}})^{1/2} \RoundBrackets*{\mathcal{R} \underline{\mathsf{B}}^\dagger \underline{\mathsf{B}} \,\underline{\mathsf{B}}^\dagger \mathcal{R}^*\mathsf{A} \mathsf{u} \cdot \mathsf{A}\mathsf{u}}^{1/2} \\
\leq & C_T^{-1/2} \mathsf{a}(\mathsf{u}, \mathsf{u})^{1/2} \RoundBrackets*{\mathcal{R} \underline{\mathsf{B}}^\dagger  \mathcal{R}^*\mathsf{A} \mathsf{u} \cdot \mathsf{A}\mathsf{u}}^{1/2} \\
= & C_T^{-1/2} \mathsf{a}(\mathsf{u}, \mathsf{u})^{1/2} \mathsf{a}(\mathcal{R} \underline{\mathsf{B}}^\dagger  \mathcal{R}^*\mathsf{A} \mathsf{u}, \mathsf{u})^{1/2},
\end{aligned}
\]
which gives that $C_T\, \mathsf{a}(\mathsf{u}, \mathsf{u}) \leq \mathsf{a}(\mathcal{R}\underline{\mathsf{B}}^\dagger\mathcal{R}^*\mathsf{A}\mathsf{u}, \mathsf{u})$ the first part of the target inequality. For the second part,
\[
\begin{aligned}
\mathsf{a}(\mathcal{R}\underline{\mathsf{B}}^\dagger\mathcal{R}^*\mathsf{A}\mathsf{u}, \mathsf{u}) = & \mathcal{R}\underline{\mathsf{B}}^\dagger\mathcal{R}^*\mathsf{A}\mathsf{u} \cdot \mathsf{A} \mathsf{u} = \sqrt{\mathsf{A}} \mathcal{R}\underline{\mathsf{B}}^\dagger\mathcal{R}^*\mathsf{A}\mathsf{u} \cdot \sqrt{\mathsf{A}} \mathsf{u} \\
\leq & \RoundBrackets*{\sqrt{\mathsf{A}} \mathcal{R}\underline{\mathsf{B}}^\dagger\mathcal{R}^*\mathsf{A}\mathsf{u} \cdot \sqrt{\mathsf{A}} \mathcal{R}\underline{\mathsf{B}}^\dagger\mathcal{R}^*\mathsf{A}\mathsf{u}}^{1/2} \RoundBrackets*{\sqrt{\mathsf{A}} \mathsf{u} \cdot \sqrt{\mathsf{A}} \mathsf{u}}^{1/2} \\
= & \mathsf{a}(\mathcal{R}\underline{\mathsf{B}}^\dagger\mathcal{R}^*\mathsf{A}\mathsf{u}, \mathcal{R}\underline{\mathsf{B}}^\dagger\mathcal{R}^*\mathsf{A}\mathsf{u})^{1/2} \mathsf{a}(\mathsf{u}, \mathsf{u})^{1/2} \\
\leq & C_R^{1/2} \mathsf{b}(\underline{\mathsf{B}}^\dagger\mathcal{R}^*\mathsf{A}\mathsf{u}, \underline{\mathsf{B}}^\dagger\mathcal{R}^*\mathsf{A}\mathsf{u})^{1/2} \mathsf{a}(\mathsf{u}, \mathsf{u})^{1/2} \\
= & C_R^{1/2} \RoundBrackets*{\underline{\mathsf{B}}^\dagger\mathcal{R}^*\mathsf{A}\mathsf{u}\cdot \underline{\mathsf{B}}\,\underline{\mathsf{B}}^\dagger\mathcal{R}^*\mathsf{A}\mathsf{u}}^{1/2} \mathsf{a}(\mathsf{u}, \mathsf{u})^{1/2} \\
= &  C_R^{1/2} \RoundBrackets*{\mathcal{R} \underline{\mathsf{B}}^\dagger \underline{\mathsf{B}}\, \underline{\mathsf{B}}^\dagger\mathcal{R}^*\mathsf{A}\mathsf{u}\cdot \mathsf{A}\mathsf{u}}^{1/2}\mathsf{a}(\mathsf{u}, \mathsf{u})^{1/2} \\
= & C_R^{1/2} \mathsf{a}(\mathcal{R} \underline{\mathsf{B}}^\dagger\mathcal{R}^*\mathsf{A}\mathsf{u}, \mathsf{u})^{1/2} \mathsf{a}(\mathsf{u}, \mathsf{u})^{1/2},
\end{aligned}
\]
and this gives $\mathsf{a}(\mathcal{R}\underline{\mathsf{B}}^\dagger\mathcal{R}^*\mathsf{A}\mathsf{u}, \mathsf{u}) \leq C_R\, \mathsf{a}(\mathsf{u}, \mathsf{u})$ due to $0< C_T\, \mathsf{a}(\mathsf{u}, \mathsf{u})\leq \mathsf{a}(\mathcal{R}\underline{\mathsf{B}}^\dagger\mathcal{R}^*\mathsf{A}\mathsf{u}, \mathsf{u})$.
\end{proof}

If take $\mathsf{b}(\cdot, \cdot)$ as
\begin{equation} \label{eq:b_bilinear}
\mathsf{b}(\underline{\mathsf{v}}, \underline{\mathsf{w}})=\underline{\mathsf{v}} \cdot \underline{\mathsf{B}} \,\underline{\mathsf{w}} \coloneqq \mathsf{v}_0\cdot \mathsf{A}_0\mathsf{w}_0 + \sum_{i=1}^{N} \mathsf{v}_i\cdot \mathsf{A}_i\mathsf{w}_i,
\end{equation}
where $\underline{\mathsf{v}}=[\mathsf{v}_0,\mathsf{v}_1,\dots,\mathsf{v}_N], \underline{\mathsf{w}}=[\mathsf{w}_0,\mathsf{w}_1,\dots,\mathsf{w}_N] \in \underline{V}$, and the map $\mathcal{R}$ as
\begin{equation} \label{eq:R}
\mathcal{R}\underline{\mathsf{v}} \coloneqq \mathsf{R}_0^\intercal \mathsf{v}_0+\sum_{i=1}^{N} \mathsf{R}_i^\intercal \mathsf{v}_i,
\end{equation}
we can then show that
\[
\underline{\mathsf{B}}^\dagger\,\underline{\mathsf{v}}=[\mathsf{A}_0^\dagger \mathsf{v}_0,\mathsf{A}_1^{-1}\mathsf{v}_1,\dots,\mathsf{A}_N^{-1}\mathsf{v}_N],\quad \forall \underline{\mathsf{v}}\in \underline{V},
\]
and
\[
\mathcal{R}^*\mathsf{v}=[\mathsf{R}_0\mathsf{v},\mathsf{R}_1\mathsf{v},\dots,\mathsf{R}_N\mathsf{v}],\quad \forall \mathsf{v} \in \Real^n.
\]
It is obvious to examine that $\mathcal{R}\underline{\mathsf{B}}^\dagger\mathcal{R}^\star\mathsf{A}=\mathsf{P}^{-1}\mathsf{A}$, and \cref{lem:fsl} tells
\[
\Cond(\mathsf{P}^{-1}\mathsf{A})\leq \frac{C_T}{C_R}
\]
given that \cref{eq:fsl1,eq:fsl2} hold, which boils our objective down to estimations of $C_T$ and $C_R$.

We have mentioned that $\mathsf{A}_i \neq \mathsf{R}_i \mathsf{A} \mathsf{R}_i^\intercal$. Actually, the relation could be further elucidated as the following lemma, where the notation $\mathsf{M} \lesssim \mathsf{N}$ for two matrices $\mathsf{M}$ and $\mathsf{N}$ means $\mathsf{N}-\mathsf{M}$ is positive semi-definite.

\begin{lemma}\label{lem:A_i}
The relation $\mathsf{R}_i \mathsf{A} \mathsf{R}_i^\intercal \lesssim \mathsf{A}_i$ holds.
\end{lemma}
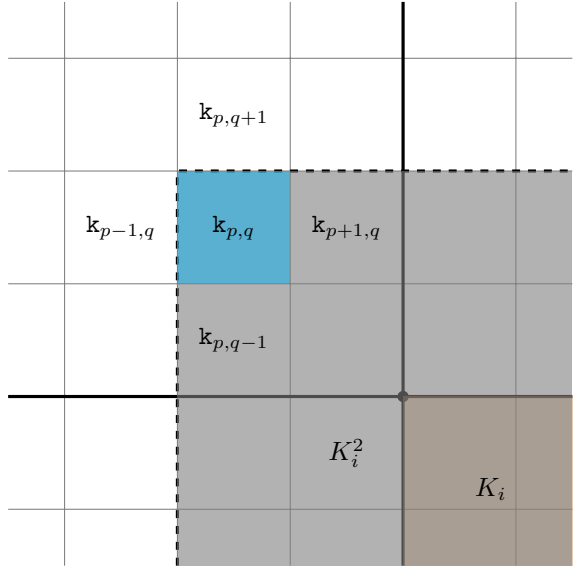
\begin{figure}[tbhp]
\centering
\begin{tikzpicture}[scale=1.5]
\draw[step=1.0, gray, thin] (-0.5, -0.5) grid (4.5, 4.5);
\draw[black, very thick] (-0.5, 1.0) -- (4.5, 1.0);
\draw[black, very thick] (3.0, -0.5) -- (3.0, 4.5);
\fill (3.0, 1.0) circle (0.05);
\draw[dashed, very thick] (1.0, -0.5) -- (1.0, 3.0) -- (4.5, 3.0);
\fill[brown, opacity=0.4] (3.0, -0.5) rectangle (4.5, 1.0);
\fill[gray, opacity=0.6] (1.0, -0.5) rectangle (4.5, 3.0);
\fill[cyan, opacity=0.5] (1.0, 2.0) rectangle (2.0, 3.0);
\node at (1.5, 2.5) {$\mathtt{k}_{p,q}$};
\node at (0.5, 2.5) {$\mathtt{k}_{p-1,q}$};
\node at (2.5, 2.5) {$\mathtt{k}_{p+1,q}$};
\node at (1.5, 1.5) {$\mathtt{k}_{p,q-1}$};
\node at (1.5, 3.5) {$\mathtt{k}_{p,q+1}$};
\node at (2.5, 0.5) {$K_i^2$};
\node[above left] at (4.0, 0.0) {$K_i$};

\end{tikzpicture}
\caption{An illustration for the proof of \cref{lem:A_i}, where $\mathtt{k}_{p,q}$, $\mathtt{k}_{p-1,q}$, $\mathtt{k}_{p+1,q}$, $\mathtt{k}_{p,q-1}$ and $\mathtt{k}_{p,q+1}$ are values of $\kappa$ on the fine elements respectively, the oversampling coarse element $K_i^m$ with $m=2$ is filled with gray and also decorated with dashed borderlines, and the fine element in the top-left corner of $K_i^m$ is highlighted.}
\label{fig:lem proof}
\end{figure}
\begin{proof}
For a fine element $\tau \in \mathcal{T}_h$ with $\tau \subset K_i^m$, which is corresponding to the $s$-th entry of a vector in $\Real^{n_i^m}$, the $s$-th row of $\mathsf{R}_i \mathsf{A} \mathsf{R}_i^\intercal$ is exactly same with the $s$-th row of $\mathsf{A}_i$ if $\tau$ does not contact with $\partial K_i^m$. This is because $\mathsf{A}_i$ is from discretizing \cref{eq:local} which shares the same law and discretization of the original system \cref{eq:orgional_equation} except for the boundary condition. Therefore, we only need to check entries of $\mathsf{A}_i$ and $\mathsf{R}_i\mathsf{A}\mathsf{R}_i^\intercal$ that are related to fine elements which contact the boundary.

The derivation will be easier to present if we consider a $2$D uniform structured mesh as in \cref{subsec:vet}. Moreover, we take the fine element in the top-left corner of $K_i^m$ to study the corresponding rows of $\mathsf{A}_i$ and $\mathsf{R}_i \mathsf{A} \mathsf{R}_i^\intercal$. We illustrate the geometric information in \cref{fig:lem proof}, where $\mathtt{k}_{p,q}$, $\mathtt{k}_{p-1,q}$, $\mathtt{k}_{p+1,q}$, $\mathtt{k}_{p,q-1}$ and $\mathtt{k}_{p,q+1}$ are values of $\kappa$ on the fine elements respectively. Let $s$ be the index of the entry related to the top-left fine element, and we can see that the $s$-th row of $\mathsf{A}_i$ differs from the $s$-th row of $\mathsf{R}_i \mathsf{A} \mathsf{R}_i^\intercal$ only at $[\mathsf{A}_i]_{s,s}$ and $[\mathsf{R}_i \mathsf{A} \mathsf{R}_i^\intercal]_{s,s}$, which can be calculated as
\[
\begin{aligned}
[\mathsf{A}_i]_{s,s} =& \RoundBrackets*{\frac{2}{1/\mathtt{k}_{p,q}}+\frac{2}{1/\mathtt{k}_{p,q}+1/\mathtt{k}_{p+1,q}}}\frac{h_y}{h_x} \\ &+\RoundBrackets*{\frac{2}{1/\mathtt{k}_{p,q}+1/\mathtt{k}_{p,q-1}}+\frac{2}{1/\mathtt{k}_{p,q}}}\frac{h_x}{h_y}, \\
[\mathsf{R}_i \mathsf{A} \mathsf{R}_i^\intercal]_{s,s} =& \RoundBrackets*{\frac{2}{1/\mathtt{k}_{p,q}+1/\mathtt{k}_{p-1,q}}+\frac{2}{1/\mathtt{k}_{p,q}+1/\mathtt{k}_{p+1,q}}}\frac{h_y}{h_x} \\
&+\RoundBrackets*{\frac{2}{1/\mathtt{k}_{p,q}+1/\mathtt{k}_{p,q-1}}+\frac{2}{1/\mathtt{k}_{p,q}+1/\mathtt{k}_{p,q+1}}}\frac{h_x}{h_y}.
\end{aligned}
\]
We can hence see that $[\mathsf{A}_i]_{s,s} \leq [\mathsf{R}_i \mathsf{A} \mathsf{R}_i^\intercal]_{s,s}$. From this case, we can see that $\mathsf{R}_i \mathsf{A} \mathsf{R}_i^\intercal - \mathsf{A}_i$ is exactly a diagonal matrix that all entries are non-negative, which essentially says $\mathsf{R}_i \mathsf{A} \mathsf{R}_i^\intercal \lesssim \mathsf{A}_i$.
\end{proof}

\begin{remark}
The proof of \cref{lem:A_i} also explains why we use $\mathsf{A}_i$ rather that $\mathsf{R}_i \mathsf{A} \mathsf{R}_i^\intercal$ for local solvers. Because constructing $\mathsf{A}_i$ needs the values of $\kappa$ in $K_i^m$, while constructing $\mathsf{R}_i \mathsf{A} \mathsf{R}_i^\intercal$ needs the values of $\kappa$ in $K_i^{m+1}$. In a message-passing model of parallel computing, each process handles a connected subdomain that consists of several coarse elements from $\CurlyBrackets{K_i}_{i=1}^N$ and $m$ layers of ghost points for inter-communications. For $\mathsf{R}_i \mathsf{A} \mathsf{R}_i^\intercal$, we need to set up $m+1$ layers,  which means more communications and will largely degrade the overall performance.
\end{remark}

We need the first assumption about the oversampling partition $\CurlyBrackets{K_i^m}_{i=1}^N$.

\paragraph{Assumption A} Suppose that $m\geq 1$ and there exists a positive constant $C_\text{os}$ such that for any $i$
\[
\Card\CurlyBrackets*{K_j^m:K_j^m \cap K_i^{m+1} \neq \varnothing, 1\leq j \leq N} \leq C_\text{os}.
\]
Note that $\mathsf{R}_j \mathsf{A} \mathsf{R}_i^\intercal$ is a zero matrix if $K_j^m \cap K_i^{m+1} = \varnothing$ and $1\leq i,j\leq N$, while $K_j^m \cap K_i^{m} = \varnothing$ cannot result in $\mathsf{R}_j \mathsf{A} \mathsf{R}_i^\intercal=\mathsf{0}$. The next lemma provides an estimation of $C_R$ in \cref{eq:fsl2}.

\begin{lemma}\label{lem:for C_R}
Under assumption A, for all $\underline{\mathsf{v}} \in \underline{V}$,
\[
\mathsf{a}(\mathcal{R}\underline{\mathsf{v}}, \mathcal{R}\underline{\mathsf{v}}) \leq \RoundBrackets*{1+4C_\text{os}} \mathsf{b}(\underline{\mathsf{v}}, \underline{\mathsf{v}}),
\]
where $\mathcal{R}$ and $\mathsf{b}(\cdot,\cdot)$ are defined in \cref{eq:R,eq:b_bilinear} respectively.
\end{lemma}
\begin{proof}
The inequality $\mathsf{a}(\mathsf{v}+\mathsf{w},\mathsf{v}+\mathsf{w}) \leq 2\RoundBrackets*{\mathsf{a}(\mathsf{v},\mathsf{v})+\mathsf{a}(\mathsf{w},\mathsf{w})}$ always holds if $\mathsf{A}$ is positive semi-definite. Then we have,
\[
\mathsf{a}(\mathcal{R}\underline{\mathsf{v}}, \mathcal{R}\underline{\mathsf{v}}) \leq 2\CurlyBrackets*{\mathsf{a}(\mathsf{R}^\intercal_0 \mathsf{v}_0, \mathsf{R}^\intercal_0 \mathsf{v}_0) + \mathsf{a}(\sum_{i=1}^{N} \mathsf{R}^\intercal_i \mathsf{v}_i, \sum_{i=1}^{N} \mathsf{R}^\intercal_i \mathsf{v}_i)}.
\]
Let $J_i\coloneqq \CurlyBrackets{j:K_j^m \cap K_i^{m+1} \neq \varnothing, 1\leq j \neq N}$. We can see that $\Card(J_i) \leq C_\text{os}$ according to Assumption A and
\[
\begin{aligned}
\mathsf{a}(\sum_{i=1}^{N} \mathsf{R}^\intercal_i \mathsf{v}_i, \sum_{i=1}^{N} \mathsf{R}^\intercal_i \mathsf{v}_i) = & \sum_{i,j=1}^{N} \mathsf{v}_i \cdot \mathsf{R}_i\mathsf{A}\mathsf{R}_j^\intercal \mathsf{v}_j = \sum_{i=1}^N \sum_{j \in J_i} \mathsf{v}_i \cdot \mathsf{R}_i\mathsf{A}\mathsf{R}_j^\intercal \mathsf{v}_j \\
\leq &  2\sum_{i=1}^N \sum_{j \in J_i} \mathsf{v}_i\cdot \mathsf{R}_i \mathsf{A}\mathsf{R}_i^\intercal \mathsf{v}_i+ \mathsf{v}_j\cdot \mathsf{R}_j \mathsf{A}\mathsf{R}_j^\intercal \mathsf{v}_j \\
= & 2C_\text{os} \sum_{i=1}^{N} \mathsf{v}_i\cdot \mathsf{R}_i \mathsf{A}\mathsf{R}_i^\intercal \mathsf{v}_i + 2\sum_{i=1}^N \sum_{j \in J_i} \mathsf{v}_j\cdot \mathsf{R}_j \mathsf{A}\mathsf{R}_j^\intercal \mathsf{v}_j \\
\leq & 2C_\text{os} \sum_{i=1}^{N} \mathsf{v}_i\cdot \mathsf{A}_i \mathsf{v}_i + 2\sum_{i=1}^N \sum_{j \in J_i} \mathsf{v}_j\cdot \mathsf{A}_j \mathsf{v}_j,
\end{aligned}
\]
where the last line is from \cref{lem:A_i}. Note that in the summation $\sum_{i=1}^N \sum_{j \in J_i}$, each index $i\in \CurlyBrackets{1,\dots,N}$ will be counted at most $C_\text{os}$ times. Therefore, we obtain
\[
\mathsf{a}(\sum_{i=1}^{N} \mathsf{R}^\intercal_i \mathsf{v}_i, \sum_{i=1}^{N} \mathsf{R}^\intercal_i \mathsf{v}_i) \leq 4C_\text{os} \sum_{i=1}^{N} \mathsf{v}_i\cdot \mathsf{A}_i \mathsf{v}_i.
\]
Recalling that $\mathsf{a}(\mathsf{R}^\intercal_0 \mathsf{v}_0, \mathsf{R}^\intercal_0 \mathsf{v}_0)=\mathsf{R}^\intercal_0 \mathsf{v}_0\cdot \mathsf{A}\mathsf{R}^\intercal_0 \mathsf{v}_0=\mathsf{v}_0\cdot \mathsf{A}_0 \mathsf{v}_0$, we hence finish the proof.
\end{proof}

We can represent $s_i(\cdot,\cdot)$ defined in \cref{eq:s_i_func} in a linear algebraic form as
\[
\mathsf{s}_i(\mathsf{v},\mathsf{w}) = \mathsf{v} \cdot \mathsf{S}_i\mathsf{w}
\]
for all $\mathsf{v}$ and $\mathsf{w}$ in $\Real^{n_i}$, where $\mathsf{S}_i$ is a diagonal matrix. For a fine element $\tau \subset K_i$ which is corresponding to the $s$-th entry of $\Real^{n_i}$, we can obtain that $[\mathsf{S}_i]_{s,s}=\tilde{\kappa}|_\tau \abs{\tau}$, where $\abs{\tau}$ is the area/volume of $\tau$. Similarly, we can also rewrite $a_i(\cdot,\cdot)$ in \cref{eq:a_i_func} as
\[
\mathsf{a}_i(\mathsf{v}, \mathsf{w}) = \mathsf{v} \cdot \hat{\mathsf{A}}_i \mathsf{w}
\]
for all $\mathsf{v}$ and $\mathsf{w}$ in $\Real^n_i$, where $\hat{\mathsf{A}}_i\in \Real^{n_i\times n_i}$. Let $\mathsf{C}_i\in \Real^{n_i\times n}$ be the restriction matrix that restricts the DoF in the whole domain to $K_i$. Thanks to the non-overlapping property of the partition $\CurlyBrackets{K_i}_{i=1}^N$, we can show that $\sum_{i=1}^{N}\mathsf{C}_i^\intercal \mathsf{C}_i=\mathsf{I}_n$, where $\mathsf{I}_n$ is the identity matrix of size $n$. We can have the following lemma, which is clear from the definitions \cref{eq:a_i_func,eq:varia vet} for a $2$D uniform structured mesh.
\begin{lemma}\label{lem:partition}
The relation $ \sum_{i=1}^N\mathsf{C}_i^\intercal \hat{\mathsf{A}}_i\mathsf{C}_i \lesssim \mathsf{A}$ holds.
\end{lemma}

According to the approximation property of eigenspaces, for any $v_i \in W_h(K_i)$, we can find a $w_i \in W_H^\text{c}(K_i)$ and $r_i \in W_h(K_i)$ such that $v_i=w_i+r_i$ with
\[
s_i(v_i-w_i, v_i-w_i)=s_i(r_i,r_i)\leq \frac{1}{\lambda_{L_i+1}}a_i(v_i, v_i),
\]
where $\lambda_{L_i+1}$ is the $(L_i+1)$-th smallest eigenvalue of the eigenproblem \cref{eq:spepb}. As illustrated in \cref{sec:gms}, by setting a proper $\tilde{\kappa}$ and choosing $L_i$ large enough, the eigenvalue $\lambda_{L_i+1}$ can be bounded below stably w.r.t.~contrast ratios. Let $\mathsf{E}_i \in \Real^{n_i^m\times n_i}$ be an extension matrix that extends the DoF in $K_i$ to $K_i^m$. Our analysis requires the following assumption.

\paragraph{Assumption B} Suppose that $\tilde{\kappa}$ is constructed such that $\mathsf{E}_i^\intercal \mathsf{A}_i \mathsf{E}_i \lesssim \mathsf{S}_i$, and $L_i$ is chosen such that $1/\lambda_{L_i+1} \leq \epsilon$ for any $i$, where $\epsilon$ is a predefined positive number.

To find a $\tilde{\kappa}$ that fulfills Assumption B, a direct way is to lump all absolute values of the off-diagonal entries into the diagonal entry in each row. Again, we explain the idea with \cref{fig:assump_B} for a $2$D uniform structured mesh. Let a test function $v$ belong to $W_h(K_i)$ that corresponds to a $\mathsf{v}\in \Real^{n_i}$. Then we can calculate that
\[
\begin{aligned}
\mathsf{v}\cdot \mathsf{E}_i^\intercal \mathsf{A}_i \mathsf{E}_i \mathsf{v} =& \kappa_{e_1}\RoundBrackets*{\mathtt{v}_{p,q}-\mathtt{v}_{p,q-1}}^2 \frac{h_y}{h_x} + \kappa_{e_2}\RoundBrackets*{\mathtt{v}_{p,q}-\mathtt{v}_{p+1,q}}^2 \frac{h_x}{h_y} + \kappa_{e_3}\mathtt{v}_{p,q}^2 \frac{h_x}{h_y} + \kappa_{e_4}\mathtt{v}_{p,q}^2 \frac{h_y}{h_x}\\
& + \CurlyBrackets*{\text{other terms}},
\end{aligned}
\]
where
\[
\begin{aligned}
\kappa_{e_1}=\frac{2}{1/\mathsf{k}_{p,q}+1/\mathsf{k}_{p,q-1}}, \quad &\kappa_{e_2}=\frac{2}{1/\mathsf{k}_{p,q}+1/\mathsf{k}_{p+1,q}}, \quad\kappa_{e_3}=\frac{2}{1/\mathsf{k}_{p,q}+1/\mathsf{k}_{p,q+1}}, \\
\quad\text{and}\quad &\kappa_{e_4}=\frac{2}{1/\mathsf{k}_{p,q}+1/\mathsf{k}_{p-1,q}}.
\end{aligned}
\]
Via the inequalities $\RoundBrackets*{\mathtt{v}_{p,q}-\mathtt{v}_{p,q-1}}^2\leq 2(\mathtt{v}_{p,q}^2+\mathtt{v}_{p+1,q}^2)$ and $\RoundBrackets*{\mathtt{v}_{p,q}-\mathtt{v}_{p,q+1}}^2\leq 2(\mathtt{v}_{p,q}^2+\mathtt{v}_{p,q+1}^2)$, we can see that
\[
\begin{aligned}
\mathsf{v}\cdot \mathsf{E}_i^\intercal \mathsf{A}_i \mathsf{E}_i \mathsf{v} \leq & \RoundBrackets*{2\kappa_{e_1}\frac{h_y}{h_x}+2\kappa_{e_2}\frac{h_x}{h_y}+\kappa_{e_3}\frac{h_y}{h_x}+\kappa_{e_3}\frac{h_x}{h_y}} \mathtt{v}_{p,q}^2 \\
&+\CurlyBrackets*{\text{other terms}}.
\end{aligned}
\]
Recalling that $\mathsf{S}_i$ is a diagonal matrix with $[\mathsf{S}_i]_{s,s}=\tilde{\kappa}|_\tau \abs{\tau}$, we can set
\begin{equation} \label{eq:tilde_kappa}
\tilde{\kappa}|_\tau = \RoundBrackets*{2\kappa_{e_1}\frac{h_y}{h_x}+2\kappa_{e_2}\frac{h_x}{h_y}+\kappa_{e_3}\frac{h_y}{h_x}+\kappa_{e_3}\frac{h_x}{h_y}}\frac{1}{h_xh_y},
\end{equation}
where $\tau$ is the fine element in the top-left corner of $K_i$. The values of $\tilde{\kappa}$ on other fine elements could also be found in this way.

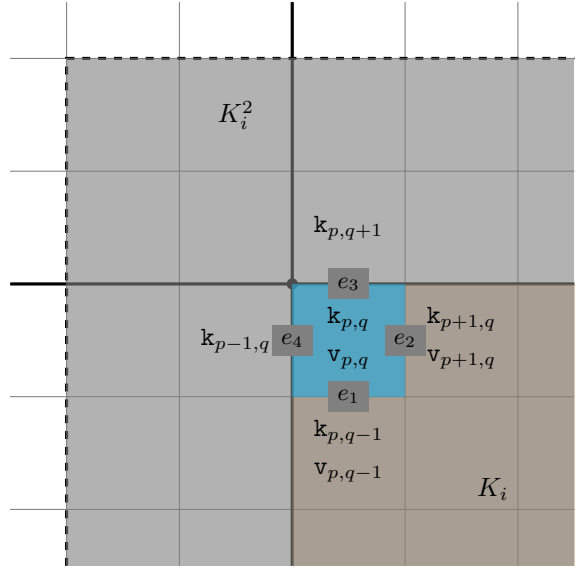
\begin{figure}[tbhp]
\centering
\begin{tikzpicture}[scale=1.5]
\draw[step=1.0, gray, thin] (0.5, -0.5) grid (5.5, 4.5);
\draw[black, very thick] (0.5, 2.0) -- (5.5, 2.0);
\draw[black, very thick] (3.0, -0.5) -- (3.0, 4.5);
\fill (3.0, 2.0) circle (0.05);
\draw[dashed, very thick] (1.0, -0.5) -- (1.0, 4.0) -- (5.5, 4.0);
\fill[brown, opacity=0.4] (3.0, -0.5) rectangle (5.5, 2.0);
\fill[gray, opacity=0.6] (1.0, -0.5) rectangle (5.5, 4.0);
\fill[cyan, opacity=0.5] (3.0, 1.0) rectangle (4.0, 2.0);
\node[above] at (3.5, 1.5) {$\mathtt{k}_{p,q}$};
\node[below] at (3.5, 1.5) {$\mathtt{v}_{p,q}$};
\node at (2.5, 1.5) {$\mathtt{k}_{p-1,q}$};
\node[above] at (4.5, 1.5) {$\mathtt{k}_{p+1,q}$};
\node[below] at (4.5, 1.5) {$\mathtt{v}_{p+1,q}$};
\node[above] at (3.5, 0.5) {$\mathtt{k}_{p,q-1}$};
\node[below] at (3.5, 0.5) {$\mathtt{v}_{p,q-1}$};
\node at (3.5, 2.5) {$\mathtt{k}_{p,q+1}$};
\node at (2.5, 3.5) {$K_i^2$};
\node[above left] at (5.0, 0.0) {$K_i$};
\node at (3.5, 1.0) [fill=gray] {\small $e_1$};
\node at (4.0, 1.5) [fill=gray] {\small $e_2$};
\node at (3.5, 2.0) [fill=gray] {\small $e_3$};
\node at (3.0, 1.5) [fill=gray] {\small $e_4$};
\end{tikzpicture}
\caption{An illustration for a construction of $\tilde{\kappa}$, where $\mathtt{k}_{p,q}$, $\mathtt{k}_{p-1,q}$, $\mathtt{k}_{p+1,q}$, $\mathtt{k}_{p,q-1}$ and $\mathtt{k}_{p,q+1}$ are values of $\kappa$ on the located fine elements respectively, the fine element in the top-left corner of $K_i$ is highlighted with four edges $e_1\sim e_4$ surrounded, $\mathtt{v}_{p,q}$, $\mathtt{v}_{p+1,q}$ and $\mathtt{v}_{p,q-1}$ are values of $v\in W_h(K_i)$ on the located fine elements respectively.}
\label{fig:assump_B}
\end{figure}

Certainly, if $\tilde{\kappa}$ satisfies Assumption B, then $\tilde{\kappa}$ is positive on each fine element, which implies $\mathsf{S}_i$ is invertible. We can use this information to depict the kernel of $\mathsf{A}_0$ as the following lemma.
\begin{lemma}\label{lem:ker A_0}
Under Assumption B, the dimension of $\ker(\mathsf{A}_0)$ is $1$, and $\mathcal{R}$ maps $\ker(\underline{\mathsf{B}})$ into $\ker(\mathsf{A})$, where $\mathcal{R}$ and $\underline{\mathsf{B}}$ are defined in \cref{eq:R,eq:b_bilinear}.
\end{lemma}
\begin{proof}
Let $\mathsf{w}\in \Real^{N^\text{c}}$ with $\mathsf{A}_0\mathsf{w}=\mathsf{0}$. Then,
\[
\mathsf{w}\cdot \mathsf{A}_0\mathsf{w}=0 \Rightarrow \mathsf{w}\cdot \mathsf{R}_0\mathsf{A}\mathsf{R}^\intercal_0 \mathsf{w}=0 \Rightarrow \sqrt{\mathsf{A}}\mathsf{R}^\intercal_0 \mathsf{w}=\mathsf{0} \Rightarrow \mathsf{R}^\intercal_0 \mathsf{w} \in \ker(\mathsf{A}).
\]
This gives us that the function $w_h \in W_h$ represented by $\mathsf{R}^\intercal_0 \mathsf{w}$ is constant in $\Omega$. On each coarse element $K_i$, the smallest eigenvalue is $0$ with a corresponding eigenvector is a constant vector. Due to $\mathsf{S}_i$ is invertible, the eigenspace corresponding to $0$ has a dimension of one. Note that $\CurlyBrackets{K_i}_{i=1}^N$ is non-overlapping, then restriction of $w_h$ on $K_i$ has only one DoF, which hence determines $\mathsf{w}$. Although that $\mathsf{w}$ may not be a constant vector, we have $\dim(\ker(\mathsf{A}_0))=1$. It is clear that $\mathcal{R}$ \cref{eq:R} maps $\ker(\underline{\mathsf{B}})$ into $\ker(\mathsf{A})$ because $\mathsf{\mathsf{A}}_i$ is invertible and $\ker(\mathsf{A}_0)$ has been fully understood.
\end{proof}

The following lemma provides an estimation of $C_T$ in \cref{eq:fsl1}.
\begin{lemma}\label{lem:for C_T}
Under Assumption A and B, for all $\mathsf{u}\in \Image(\mathsf{A})$, there exists a $\underline{\mathsf{v}}\in \Image(\underline{\mathsf{B}})$ with $\mathsf{u}=\mathcal{R}\underline{\mathsf{v}}$ and
\[
\frac{1}{2+(8C_\text{os}+1)\epsilon}\,\mathsf{b}(\underline{\mathsf{v}}, \underline{\mathsf{v}}) \leq \mathsf{a}(\mathsf{u}, \mathsf{u}),
\]
where $\underline{\mathsf{B}}$, $\mathcal{R}$ and $\mathsf{b}(\cdot,\cdot)$ are defined in \cref{eq:R,eq:b_bilinear}.
\end{lemma}
\begin{proof}
We decompose $\mathsf{C}_i \mathsf{u}$ as $\mathsf{C}_i\mathsf{u}=\mathsf{w}_i+\mathsf{r}_i$, where the function representation of $\mathsf{w}_i$ belongs to $W_H^\text{c}(K_i)$. According to the approximation property of eigenspaces, we can find a $\mathsf{w}_i\in \Real^{n_i}$ such that $\mathsf{s}_i(\mathsf{r}_i, \mathsf{r}_i)$ is minimized with an estimation
\begin{equation}\label{eq:C_T eigen esti}
\mathsf{s}_i(\mathsf{r}_i, \mathsf{r}_i)=\mathsf{r}_i \cdot \mathsf{S}_i\mathsf{r}_i \leq \epsilon \mathsf{a}_i(\mathsf{C}_i \mathsf{u}, \mathsf{C}_i \mathsf{u})= \epsilon \mathsf{u}\cdot \mathsf{C}_i^\intercal \hat{\mathsf{A}}_i\mathsf{C}_i \mathsf{u}.
\end{equation}
Recalling that $\mathsf{u}=\sum_{i=1}^{N}\mathsf{C}_i^\intercal \mathsf{C}_i \mathsf{u}$, we have
\[
\mathsf{u}=\sum_{i=1}^{N}\mathsf{C}_i^\intercal \RoundBrackets*{\mathsf{w}_i + \mathsf{r}_i} = \sum_{i=1}^{N} \mathsf{C}_i^\intercal\mathsf{w}_i + \sum_{i=1}^{N}\mathsf{R}_i^\intercal \mathsf{r}_i',
\]
where $\mathsf{r}_i'=\mathsf{E}_i\mathsf{r}_i$. Moreover, from constructions of the global coarse space $W_H^\text{c}$ and $\mathsf{R}_0$, there exists a $\mathsf{w}\in \Real^{N^\text{c}}$ such that $\sum_{i=1}^{N} \mathsf{C}_i^\intercal\mathsf{w}_i=\mathsf{R}_0^\intercal \mathsf{w}$. We hence derive a vector $\underline{\mathsf{v}}\in \underline{V}$ as
\[
\underline{\mathsf{v}} =[\mathsf{w}, \mathsf{r}_1',\dots,\mathsf{r}_N']
\]
that satisfies $\mathsf{u}=\mathcal{R}\underline{\mathsf{v}}$. Thanks to \cref{lem:ker A_0}, it is easy to check that if $\mathsf{u}\in \Image(\mathsf{A})$, the constructed vector $\underline{\mathsf{v}}$ belongs to $\Image(\underline{\mathsf{B}})$.

Note that
\[
\mathsf{b}(\underline{\mathsf{v}}, \underline{\mathsf{v}})=\mathsf{w}\cdot \mathsf{A}_0\mathsf{w}+\sum_{i=1}^{N}\mathsf{r}_i'\cdot \mathsf{A}_i\mathsf{r}_i'=\mathsf{a}(\mathsf{R}_0^\intercal \mathsf{w}, \mathsf{R}_0^\intercal \mathsf{w}) + \sum_{i=1}^{N}\mathsf{r}_i\cdot \mathsf{E}_i^\intercal\mathsf{A}_i\mathsf{E}_i\mathsf{r}_i.
\]
We first need to bound $\mathsf{a}(\mathsf{R}_0^\intercal \mathsf{w}, \mathsf{R}_0^\intercal \mathsf{w})$. Employing techniques in the proof of \cref{lem:for C_R}, we have
\[
\begin{aligned}
\mathsf{a}(\mathsf{R}_0^\intercal \mathsf{w}, \mathsf{R}_0^\intercal \mathsf{w}) \leq& 2 \CurlyBrackets*{\mathsf{a}(\mathsf{u}, \mathsf{u})+\mathsf{a}(\sum_{i}^N \mathsf{R}_i^\intercal \mathsf{r}_i', \sum_{i}^N \mathsf{R}_i^\intercal \mathsf{r}_i')} \\
\leq& 2\CurlyBrackets*{\mathsf{a}(\mathsf{u}, \mathsf{u})+4C_\text{os}\sum_{i=1}^{N}\mathsf{r}_i'\cdot \mathsf{A}_i\mathsf{r}_i'} \\
=&2\CurlyBrackets*{\mathsf{a}(\mathsf{u}, \mathsf{u})+4C_\text{os}\sum_{i=1}^{N}\mathsf{r}_i\cdot \mathsf{E}_i^\intercal\mathsf{A}_i\mathsf{E}_i\mathsf{r}_i}.
\end{aligned}
\]
Combining Assumption B with \cref{eq:C_T eigen esti}, we arrive at
\[
\begin{aligned}
\sum_{i=1}^{N}\mathsf{r}_i\cdot \mathsf{E}_i^\intercal\mathsf{A}_i\mathsf{E}_i\mathsf{r}_i\leq & \sum_{i=1}^{N} \mathsf{r}_i \cdot \mathsf{S}_i\mathsf{r}_i \leq \epsilon\sum_{i=1}^{N} \mathsf{u}\cdot \mathsf{C}_i^\intercal \hat{\mathsf{A}}_i \mathsf{C}_i \mathsf{u} \\
\leq & \epsilon \mathsf{u}\cdot \mathsf{A}\mathsf{u} = \epsilon \mathsf{a}(\mathsf{u}, \mathsf{u}),
\end{aligned}
\]
where the last line is from \cref{lem:partition}. Therefore, we derive that
\[
\mathsf{b}(\underline{\mathsf{v}}, \underline{\mathsf{v}}) \leq 2\CurlyBrackets*{1+4C_\text{os}\epsilon}\mathsf{a}(\mathsf{u},\mathsf{u})+\epsilon\mathsf{a}(\mathsf{u},\mathsf{u}) \leq \CurlyBrackets*{2+(8C_\text{os}+1)\epsilon}\mathsf{a}(\mathsf{u},\mathsf{u}),
\]
which leads the completion of the proof.
\end{proof}

Combining \cref{lem:for C_R,lem:for C_T,lem:ker A_0}, we reach the main theorem of this section.
\begin{theorem}\label{thm:main}
Under Assumption A and B, the inequality
\[
\Cond(\mathsf{P}^{-1}\mathsf{A}) \leq \RoundBrackets*{2+(8C_\text{os}+1)\epsilon}\RoundBrackets*{1+4C_\text{os}}
\]
holds.
\end{theorem}

\section{Numerical experiments}\label{sec:num}
We implement our method with PETSc \cite{Balay2022a}. We set the unit cube $(0, 1)^3$ as the computational domain with uniform structured meshes for all experiments (expect for the SPE10 problem). The domain decomposition is handled by PETSc's DMDA module, i.e., each MPI process owns a non-overlapping rectangular cuboid region plus $m$ layers of neighboring ghost points for inter-process communications. Meanwhile, each cuboid is further divided into $\mathtt{sd}$ parts in $x$-, $y$- and $z$-direction, which gives a non-overlapping decomposition of the whole domain into total $\mathtt{sd}^3 \times \mathtt{proc}$ subdomains (also called coarse elements), where $\mathtt{proc}$ stands for the number of MPI processes, and this also signifies $N=\mathtt{sd}^3 \times \mathtt{proc}$. Note that in our computing settings, coarse elements are not of the same size, whilst $H$ the minimal size could be roughly estimated as $1.0 / (\mathtt{sd} \times \sqrt[3]{\mathtt{proc}})$. In each coarse element, we solve the eigenvalue problem \cref{eq:spepb} with SLEPc \cite{Hernandez2005} and obtain $L_\star$ eigenvectors. For local and global solvers, we take CHOLMOD \cite{Chen2008} to factorize $\mathsf{A}_i$ and MUMPS \cite{Amestoy2001} to factorize $\mathsf{A}_0$.

We place $4$ long singular sources at $4$ corner points on $x\times y$-plane and an opposite sign long singular source at the middle of $x\times y$-plane as the source term $f$ of the model \cref{eq:orgional_equation}, which mimics the well condition in the SPE10 model (ref. \cite{Christie2001}). We set the default parameters for the GMRES solver provided by PETSc to solve the linear algebraic system \cref{eq:fine_system}, and the convergence criterion is that $l^2$-norm of the residual reaches $10^{-5}$. All computational elapsed time records in numerical experiments are measured from an HPC cluster that each computing node is configured with dual Intel\textsuperscript{\textregistered} Xeon\textsuperscript{\textregistered} Gold 6140 CPUs (36 cores) and 192GB memory. In the following reports, We denote by $\mathtt{iter}$ the number of GMRES iterations, $\mathtt{DoF}$ the number of total DoF.

We address several different settings in numerical experiments compared with Assumption B in \cref{sec:anal}: first, we preset a fixed number $L_\star$ of eigenvectors for all eigenproblems; second, we simply choose $\tilde{\kappa}=\kappa$ rather than the sophisticated construction in \cref{eq:tilde_kappa}. We explain those modifications here. An important topic in designing parallel algorithms is the so-called load balancing. In our situation, \cref{thm:main} implies that determining $L_i$ the number of eigenvectors adaptively according to the threshold for eigenvalues may achieve a theoretical bound for the condition number. However, for some challenging permeability profiles, the variance of $\CurlyBrackets{L_i}_{i=1}^N$ may be considerably large, which causes a load imbalance among processes and will tremendously degrade the overall performance. Moreover, thanks to the non-overlapping property, we could directly construct $L_\star$ global vectors defined on the fine mesh for storing all eigenvectors (coarse bases). Then, we can directly leverage the functionality from DMDA for handling communications in building the global solver $\mathsf{A}_0$, which greatly facilitates our implementation. For choosing $\tilde{\kappa}=\kappa$, recalling in \cref{eq:tilde_kappa}, we obtain an expression for $\tilde{\kappa}|_\tau$ with $\kappa_{e_1}$, $\kappa_{e_2}$, $\kappa_{e_3}$ and $\kappa_{e_4}$, and we can show that $\kappa_{e_1}$, $\kappa_{e_2}$, $\kappa_{e_3}$ and $\kappa_{e_4}$ can all be bounded by $2\mathtt{k}_{p,q}$ which is $2\kappa|_\tau$. This observation implies that by taking $\tilde{\kappa}=C\kappa$ with a modest large constant $C$, the inequality $\mathsf{E}_i^\intercal \mathsf{A}_i \mathsf{E}_i \lesssim \mathsf{S}_i$ is obtainable. Since we only require a fixed number of eigenvectors, the constant $C$ does not affect the resulting eigenspace (also called the local coarse space). Again, taking $\tilde{\kappa}=\kappa$ is easier to implement in contrast to \cref{eq:tilde_kappa}, because the latter needs to cook up different formulas depending on whether the fine element touches $\partial K_i$.

\subsection{Contrast robustness}
In this subsection, we will test our preconditioner with two types of media,  that is, one contains long channels and the other contains fractures. We will fix the coefficient $\kappa=1$ in the background region, while increasing values of $\kappa$ in channels and fractures to examine robustness w.r.t. contrast ratios. All models have the same resolution as $\mathtt{DoF}=512^3$ and are solved with the same number of MPI processes as $\mathtt{proc}=8^3$.

In the first type of media, we place long channels that penetrate the whole domain in $y$-direction, that is $\kappa(x, y, z)=\kappa(x, 0, z)$ for all $y$. The resolution of $x\times z$-plane is of $512\times 512$, and we arrange periodic cells that are in a size of $16\times 16$ along $x$- and $z$-direction. The permeability field $\kappa$ is determined by two parameters $\alpha$ and $\mathtt{cr}$ with $\kappa=10^{\alpha\times \mathtt{cr}}$, where distributions of $\alpha$ in a periodic cell are demonstrated in the left part of \cref{fig:channels}. We consider four different configurations---``2-ch. cfg.'', ``3-ch. cfg.'', ``4-ch. cfg.'' and ``5-ch. cfg.''---which are named by the number of channels inside a periodic cell. The number of GMRES iterations $\mathtt{iter}$ against $L_\star$ under different $\mathtt{cr}$ is plotted in the right part of \cref{fig:channels}, and note that we set $m=3$ and $\mathtt{sd}=4$ for those tests. From the graphs, we could see that $\mathtt{iter}$ can remain robust by enriching coarse spaces. For example, for each configuration, when $L_\star=6$ or $7$, all values of $\mathtt{iter}$ are all close with contrast ratios ranging from $10^3\sim 10^6$. Moreover, we could observe that problems are tougher with more channels (e.g., for ``4-ch. cfg.'' and ``5-ch. cfg.'' with $\mathtt{cr}=12$, the solver cannot finish within $1000$ iterations). Again, our preconditioner can achieve robust performance by incorporating a modest number of eigenvectors ($L_\star=6$) into coarse spaces.

\begin{figure}[tbhp]
\centering
\includegraphics[width=\textwidth]{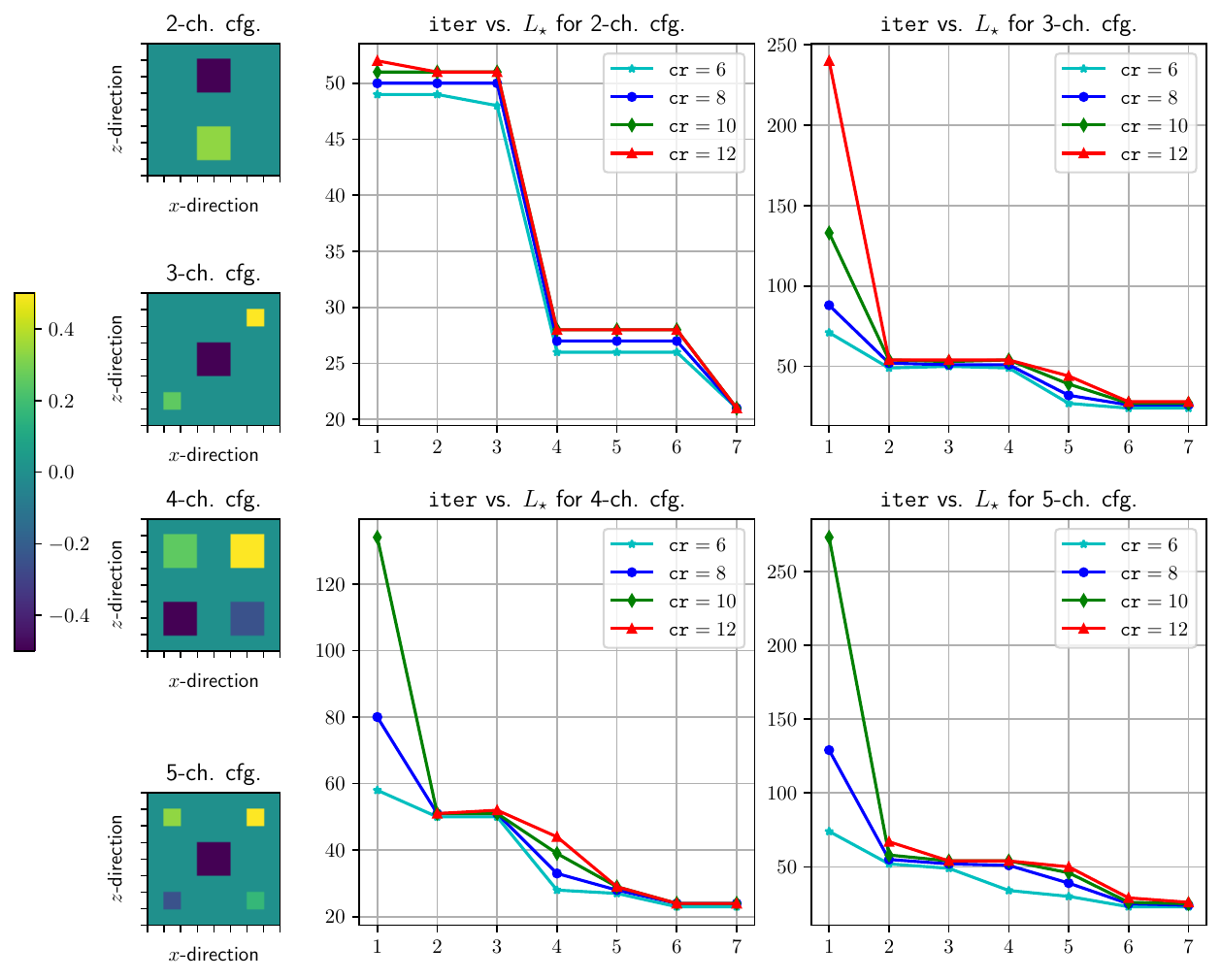}
\caption{(\textbf{left part}) the distributions of $\alpha$ for different configurations in a periodic cell, (\textbf{right part}) The number of GMRES iterations $\mathtt{iter}$ against $L_\star$ under different $\mathtt{cr}$.}
\label{fig:channels}
\end{figure}

The second type of media are visualized in the left part of \cref{fig:fracture} that contains two configurations, and the original models are in a resolution of $256^3$ while we periodically duplicate them into $512^3$ the resolution of computational domains. Note that in \cref{fig:fracture}, blue regions refer to fractures, where we will place significantly higher permeability as $\kappa=10^\mathtt{cr}$ in to reflect that flows move fast in fractures, and we simply fixed $\kappa=1$ in red regions. We emphasized that those models are, indeed, challenging to solve. To support this assertion, we adopt ``gamg''---the default algebraic multigrid preconditioner by PETSc---to GMRES and record $\mathtt{iter}$ and elapsed times in \cref{tab:gamg}. It is found that the performance of ``gamg'' gradually deteriorates as increasing $\mathtt{cr}$, and there are errors of breakdown appearing when $\mathtt{cr}=6$ and $8$. However, our preconditioner exhibits robustness w.r.t. $\mathtt{cr}$ from the data in \cref{tab:gamg} with a set of parameters $(L_\star, m, \mathtt{sd})=(4, 2, 4)$. The right part of \cref{fig:fracture} shows $\mathtt{iter}$ against $L_\star$ under different $\mathtt{cr}$ for two fracture configurations with $m=3$ and $\mathtt{sd}=4$, where the contrast ratios in those tests range from $1$ to $10^{10}$. The plots again illustrate that the number of iterations decreases as expected with increasing the dimension of coarse spaces, and a modest large $L_\star$ is necessary to achieve robustness for those high-contrast problems.

\begin{figure}[tbhp]
\centering
\includegraphics[width=\textwidth]{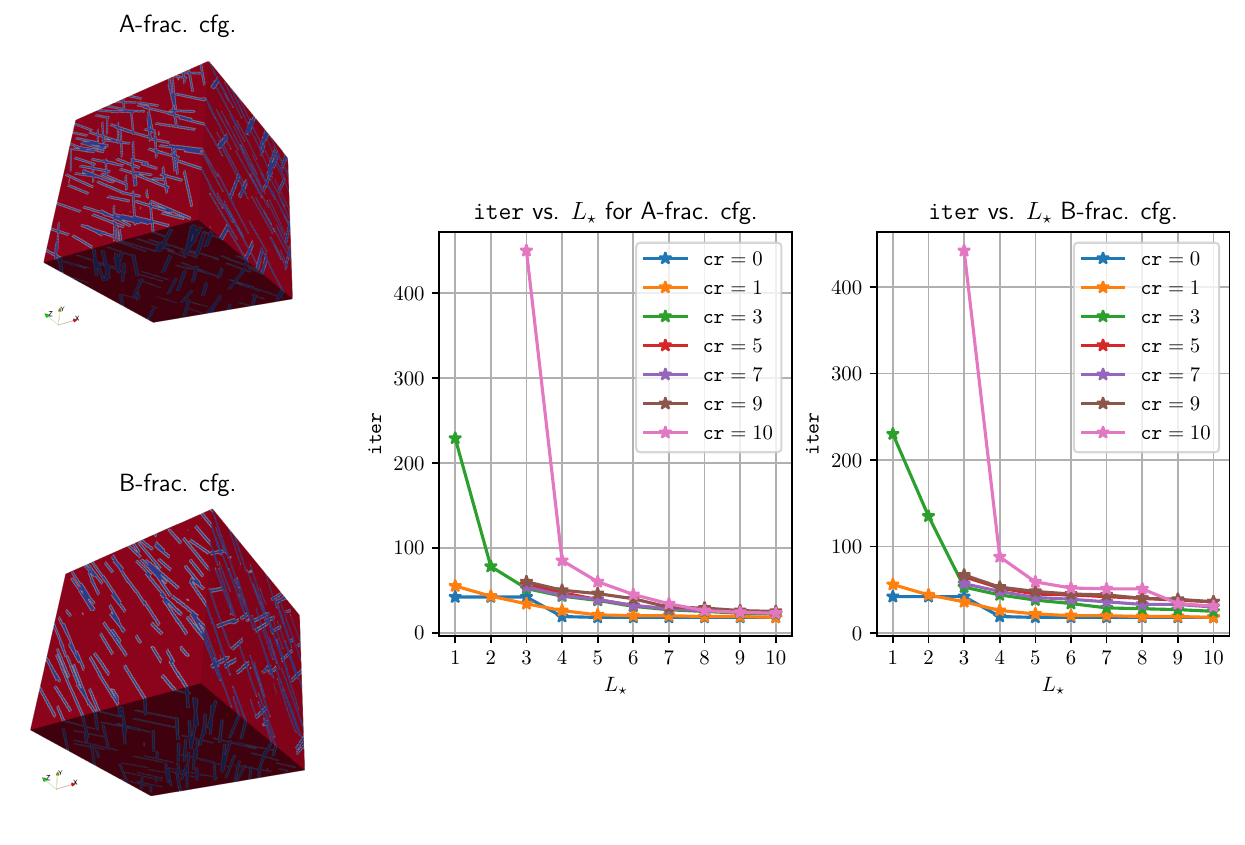}
\caption{(\textbf{left part}) two binary-valued data sets of a size of $256^3$, where the regions of blue color represent fractures, and the whole computational domain $512^3$ is constituted by periodically duplicating a $256^3$ configuration; (\textbf{right part}) plots of $\mathtt{iter}$ w.r.t. different $\mathtt{en}$ under different $\mathtt{cr}$.}
\label{fig:fracture}
\end{figure}

\begin{table}[tbhp]
\footnotesize
\caption{Iteration numbers and elapsed times (seconds) in a format $\mathtt{iter}(\text{time})$ w.r.t. different $\mathtt{cr}$, where ``gamg'' stands for PETSc's default algebraic multigrid preconditioner and ``$-$'' for breakdown error occurring, and the coefficient profile used is from ``B-frac. cfg.'' in \cref{fig:fracture}.}
\label{tab:gamg}
\centering
\begin{tabular}{|c|c|c|c|c|c|}
\hline
$\mathtt{cr}$ & $0$ & $2$ & $4$ & $6$ & $8$ \\
\hline
gamg & $9(27.8)$ & $18(33.1)$ & $99(77.8)$ & $-$ & $-$ \\
\hline
$(L_\star, m, \mathtt{sd})=(4, 2, 4)$ & $23(40.8)$ & $44(61.1)$ & $55(77.2)$ & $60(81.5)$ & $55(86.8)$ \\
\hline
\end{tabular}
\end{table}

\subsection{Scalability}
In this subsection, we fix $\mathtt{sd}=4$, $m=3$ and $L_\star=4$ while change $\mathtt{DoF}$ and $\mathtt{proc}$. Moreover, the permeability fields are generated from the periodical cell ``2-ch. cfg.'' in \cref{fig:channels} with $\mathtt{cr}=6$. Compared to \cref{alg:preconditioner}, we divide the elapsed time of obtaining solutions into three parts: Pre-0 (construct and factorize $\mathsf{A}_i$, solve local eigenvalue problems), Pre-1 (construct and factorize $\mathsf{A}_0$) and Ite  (the iteration phase). The results of strong and weak scalability tests are presented in the left and right part of \cref{fig:scal} respectively. For examining strong scalability, we set $\mathtt{DoF}=512^3$ and increase $\mathtt{proc}$ from $6^3$ to $10^3$, and a three-fold acceleration can be observed. Recall that our domain decomposition is related to the number of processes and $H \sim 1.0 / (\mathtt{sd} \times \sqrt[3]{\mathtt{proc}})$, we can also notice a decrease in $\mathtt{iter}$ with a large $\mathtt{proc}$, and this phenomenon agrees with traditional theories that a smaller $H$ can result in a better condition number. To test weak scalability, we fix the ratio between $\mathtt{DoF}$ and $\mathtt{proc}$, and the performance of $\mathtt{DoF}=640^3$ degrades to $70\%$ of $\mathtt{DoF}=256^3$, while $\mathtt{iter}$ remains unchanged. The bottleneck of our method for obtaining better scalability performance is global solvers, or more precisely, factorize $\mathsf{A}_0$. We could see this argument more clearly from the elapsed times in the Pre-1 rows in \cref{fig:scal}, where both cases show a dramatic increase with more MPI processes. Tuning direct solvers for sparse linear systems is a well-known difficulty \cite{wang2011}, and a remedy (also in our future work) could be replacing the factorization of $\mathsf{A}_0$ with an in-exact but scalable solver.

\begin{figure}[tbhp]
\centering
\includegraphics[width=\textwidth]{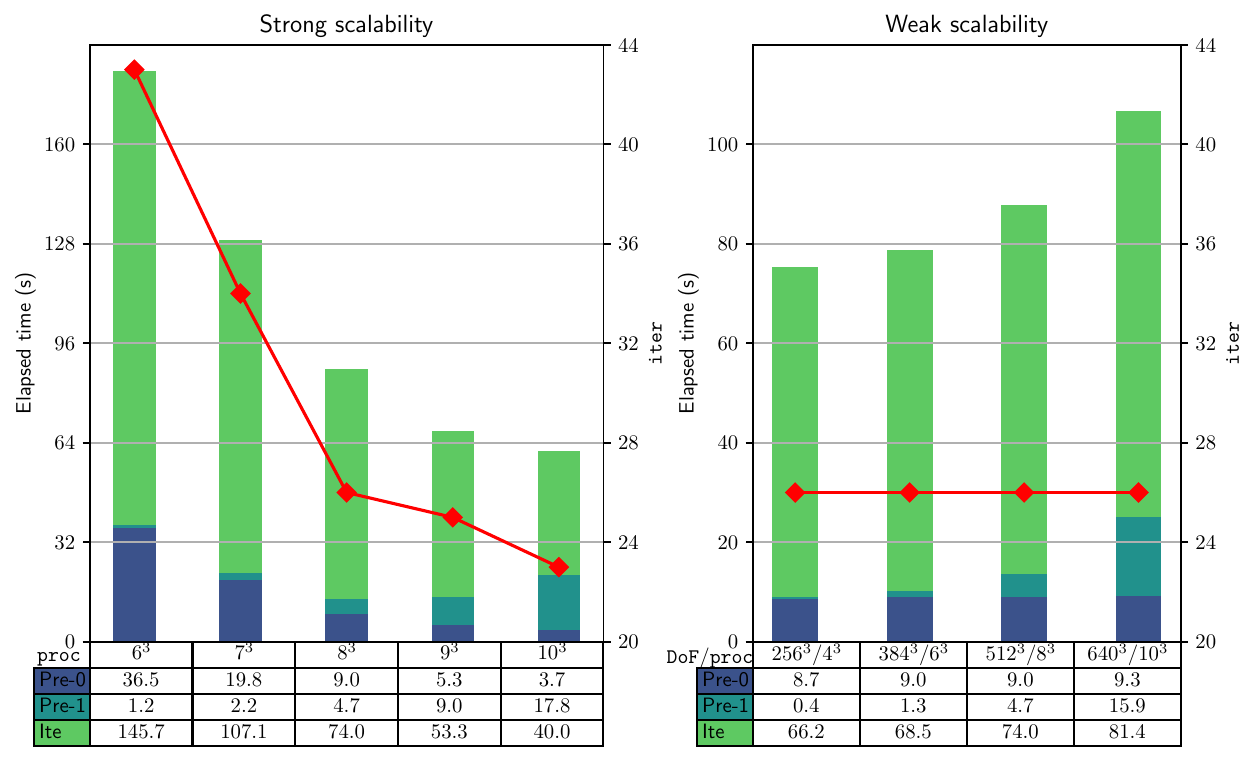}
\caption{(\textbf{left part}) elapsed times and iteration numbers in strong scalability tests with $\mathtt{DoF}=512^3$, (\textbf{right part}) elapsed times and iteration numbers in weak scalability tests with a fixed ratio between $\mathtt{DoF}$ and $\mathtt{proc}$.}
\label{fig:scal}
\end{figure}

\subsection{Parameter tests}
In this subsection, we fix $\mathtt{DoF}=512^3$, $\mathtt{proc}=8^3$ to study the influence of $\mathtt{sd}$, $m$ and $L_\star$ on the solver. The permeability field $\kappa$ is constructed from ``B-frac. cfg.'' in \cref{fig:fracture} with $\mathtt{cr}=6$. Although \cref{thm:main} does not depict how the number of oversampling layers $m$ affects the condition number, \cref{tab:diff m} indeed confirms that $\mathtt{iter}$ could be smaller with a larger $m$. However, as we have mentioned before, a larger $m$ means more communications and will degrade the final performance, which could also be validated from the elapsed times in \cref{tab:diff m} that the best time record occurs with $m=1$ and $L_\star=4$.

\begin{table}[tbhp]
\footnotesize
\caption{Iteration numbers and elapsed times (seconds) in a format $\mathtt{iter}(\text{time})$ w.r.t. different $m$ and $L_\star$ with $\mathtt{sd}=4$. }
\label{tab:diff m}
\centering
\begin{tabular}{|c|c|c|c|c|c|}
\hline
$L_\star$ & $3$ & $4$ & $5$ & $6$ & $7$ \\
\hline
$m=1$ & $107(89.6)$ & $83(80.9)$ & $80(87.3)$ & $73(93.7)$ & $58(107.5)$ \\
\hline
$m=2$ & $78(94.7)$ & $60(81.5)$ & $54(100.0)$ & $50(86.7)$ & $47(91.3)$ \\
\hline
$m=3$ & $68(120.9)$ & $53(105.6)$ & $47(102.6)$ & $44(103.5)$ & $42(110.3)$ \\
\hline
$m=4$ & $56(160.6)$ & $48(138.2)$ & $42(130.8)$ & $39(134.8)$ & $37(134.1)$ \\
\hline
\end{tabular}
\end{table}

Note that in our settings, we introduce $\mathtt{sd}$ to further divide the computational domain owned by each process. If we just set $\mathtt{sd}=1$ for restricting the time consumed on local solvers, we are left to use more MPI processes, which usually is unaffordable due to limited computational resources. The results of experiments on $\mathtt{sd}$ are presented in \cref{tab:diff sd}. Just as the strong scalability test in \cref{fig:scal}, a larger $\mathtt{sd}$ means a smaller $H$, which decreases the number of iterations. However, a larger $\mathtt{sd}$ can increase the size of $\mathsf{A}_0$, and the final performance will be burdened with global solvers. The key observation  from \cref{tab:diff sd} is that to efficiently implement two-level preconditioners, balancing the efforts on local and global solvers is crucial.

\begin{table}[tbhp]
\footnotesize
\caption{Iteration numbers and elapsed times (seconds) in a format $\mathtt{iter}(\text{time})$ w.r.t. different $\mathtt{sd}$ and $L_\star$ with $m=3$.}
\label{tab:diff sd}
\centering
\begin{tabular}{|c|c|c|c|c|c|}
\hline
$L_\star$ & $3$ & $4$ & $5$ & $6$ & $7$ \\
\hline
$\mathtt{sd}=1$ & $-$ & $173(495.4)$ & $153(510.9)$ & $138(442.0)$ & $133(429.0)$ \\
\hline
$\mathtt{sd}=2$ & $-$ & $130(186.6)$ & $127(189.4)$ & $68(123.1)$ & $64(121.1)$ \\
\hline
$\mathtt{sd}=4$ & $68(120.9)$ & $53(105.6)$ & $47(102.6)$ & $44(103.5)$ & $42(110.3)$ \\
\hline
$\mathtt{sd}=8$ & $28(185.0)$ & $24(202.1)$ & $22(268.1)$ & $22(273.1)$ & $22(331.3)$ \\
\hline
\end{tabular}
\end{table}

Next, we consider the SPE10 problem \cite{Christie2001}, which is a widely used benchmark model for testing various model reduction methods, and the adopted permeability field is visualized in the left part of \cref{fig:spe10}. 
% Note that comparing with the original SPE10 model which contains an anisotropic permeability field as $[\kappa_x,\kappa_y,\kappa_z]$, we choose $\kappa_x$ here and scale it such that its minimal value is equal to $1$.
This model exhibits  strong heterogeneity with a magnitude of $10^7$. The computed pressure field is demonstrated in the right part of \cref{fig:spe10}. The $\mathtt{DoF}$ in this model is around $1$ million, which is relatively small in comparison with previous experiments, and we just report several results with different $L_\star$ in \cref{tab:spe10}. We could observe again that a larger $L_\star$ can bring a smaller $\mathtt{iter}$, while the elapsed times may get improved by exploring other settings of $\mathtt{proc}$, $m$ and $\mathtt{sd}$.

\begin{figure}[tbhp]
\centering
\includegraphics[width=\textwidth]{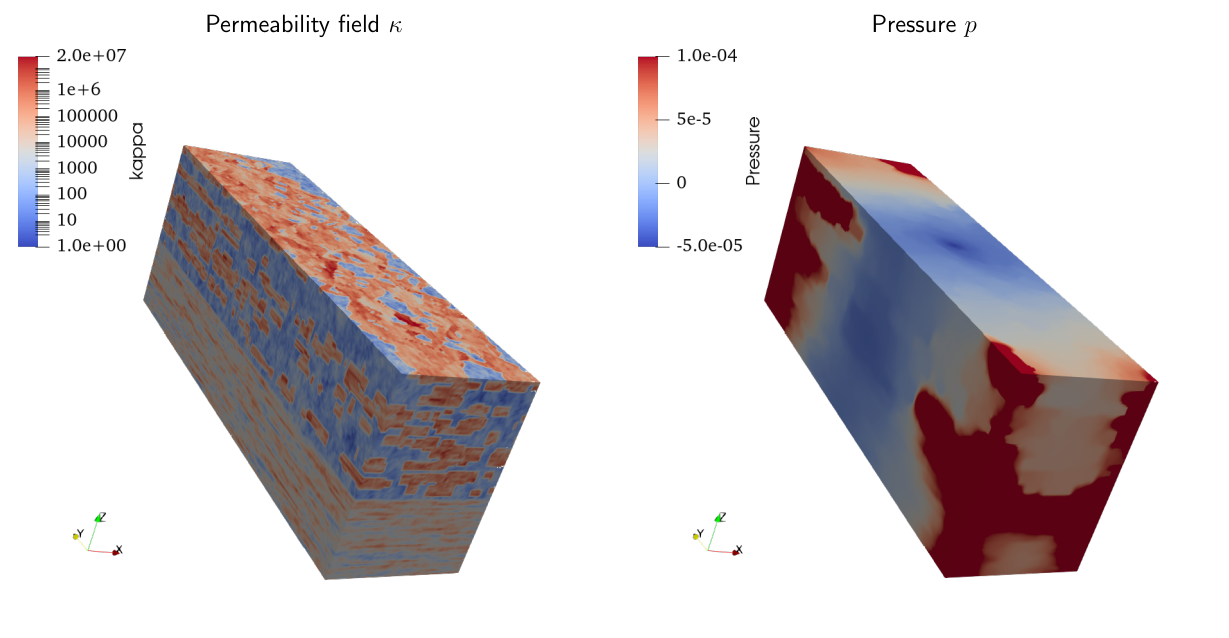}
\caption{(\textbf{left part}) the coefficient profile of the SPE10 model, (\textbf{right part}) The pressure distribution.}
\label{fig:spe10}
\end{figure}

\begin{table}[tbhp]
\footnotesize
\caption{Iteration numbers and elapsed times (seconds) in a format $\mathtt{iter}(\text{time})$ w.r.t. different $L_\star$ for the SPE10 model with $\mathtt{proc}=56$, $m=3$ and $\mathtt{sd}=2$.}
\label{tab:spe10}
\centering
\begin{tabular}{|c|C{3.5em}|C{3.5em}|C{3.5em}|C{3.5em}|C{3.5em}|C{3.5em}|C{3.5em}|}
\hline
$L_\star$ & $1$ & $2$ & $3$ & $4$ & $5$ & $6$ & $7$ \\
\hline
$\mathtt{iter}(\text{wall time})$ & $ - $ & $ - $ & $ - $ & $49(6.5)$ & $39(5.4)$ & $33(4.8)$ & $29(4.3)$ \\
\hline
\end{tabular}
\end{table}

\section{Conclusion}
In this paper, we developed a two-level overlapping preconditioner for accelerating iterative solvers for linear algebraic systems derived from discretizing Darcy flow. To prepare for practical two-phase flow simulations, the discretization that we focus on is based on the mixed form that is superior in mass conservation compared with the common second-order form, and the no-flux boundary condition is also considered rather than for example, the Dirichlet boundary condition. Inheriting the methodology of GMsFEMs, we solve general eigenvalue problems on subdomains to obtain leading eigenvectors for constructing the coarse space. Our construction distinguishes itself from existing literature by working on non-overlapping coarse elements, which to the largest extent saves communications and guarantees parallel efficiency. Our manner of building local solvers is not a direct copy of additive Schwarz methods on Galerkin discretization, and actually, it has been explained to be suitable for parallel implementations. To handle the singularity from the no-flux boundary condition, we tweaked the fictitious space lemma and obtained an a priori estimation for the conditioner number of the preconditioner under several assumptions. Our implementation is facilitated by PETSc, and numerical reports support the claim that the proposed preconditioner is robust with high contrast coefficient profiles. Furthermore, the strong and weak scalability tests inspire a possible improvement---replacing the factorization in the global solver with an in-exact but scalable one--that will be provided in our future work.

\bibliographystyle{siamplain}
\bibliography{refs}
\end{document}

%% file: SIAM_shared.tex
\usepackage{amsfonts}
\usepackage{graphicx}
\usepackage{epstopdf}
\ifpdf
  \DeclareGraphicsExtensions{.eps,.pdf,.png,.jpg}
\else
  \DeclareGraphicsExtensions{.eps}
\fi

\newsiamremark{remark}{Remark}
\newsiamremark{hypothesis}{Hypothesis}
\crefname{hypothesis}{Hypothesis}{Hypotheses}
\newsiamthm{claim}{Claim}

\headers{A robust two-level overlapping preconditioner}{C. Ye, S. Fu, E. Chung, and J. Huang}

\title{A robust two-level overlapping preconditioner for Darcy flow in high-contrast media\thanks{Submitted to the editors DATE.
\funding{
Jizu Huang's work is supported by National Key Research and Development Project of China (No. 2023YFA1011705) and National Natural Science Foundation of China (No. 12131002).
Eric Chung's work is partially supported by the Hong Kong RGC General Research Fund (Project numbers: 14305222 and 14304021). 
Shubin Fu's work is supported by start up funding of Eastern Institute for Advanced Study.
}}}

\author{Changqing Ye\thanks{Department of Mathematics, The Chinese University of Hong Kong, Shatin, Hong Kong SAR 
  (\email{cqye@math.cuhk.edu.hk}).}
\and Shubin Fu\thanks{Eastern Institute for Advanced Study, Ningbo, China 
  (\email{shubinfu89@gmail.com}, \email{shubinfu@eias.ac.cn}).}
\and Eric T. Chung\thanks{Department of Mathematics, The Chinese University of Hong Kong, Shatin, Hong Kong SAR (\email{tschung@math.cuhk.edu.hk}).}
\and Jizu Huang\thanks{LSEC, Academy of Mathematics and Systems Science, Chinese Academy of Sciences, Beijing 100190, China, and School of Mathematical Sciences, University of Chinese Academy of Sciences, Beijing 100049, China (\email{huangjz@lsec.cc.ac.cn}).}}

\usepackage{amsopn}

\usepackage{amssymb}
\usepackage{bm} % Bold math font.
\usepackage{mathtools}
\graphicspath{{figs/}} % Setting the graphicspath

\usepackage{tikz}

\newcommand{\dx}{\,\mathrm{d}}
\newcommand{\Real}{\mathbb{R}}

\DeclareMathOperator{\Div}{div}

\DeclareMathOperator{\Span}{span}
\DeclareMathOperator{\Cond}{cond}
\DeclareMathOperator{\Image}{im}
\DeclareMathOperator{\Card}{card}

\DeclarePairedDelimiter{\RoundBrackets}{(}{)}
\DeclarePairedDelimiter{\CurlyBrackets}{\{}{\}}

\usepackage{stmaryrd}
\DeclarePairedDelimiter{\DSquareBrackets}{\llbracket}{\rrbracket}

\DeclarePairedDelimiter{\abs}{\lvert}{\rvert}

\usepackage{array}
\newcolumntype{L}[1]{>{\raggedright\let\newline\\\arraybackslash\hspace{0pt}}m{#1}}
\newcolumntype{C}[1]{>{\centering\let\newline\\\arraybackslash\hspace{0pt}}m{#1}}
\newcolumntype{R}[1]{>{\raggedleft\let\newline\\\arraybackslash\hspace{0pt}}m{#1}}

\usepackage{algorithm}
\usepackage{algcompatible}
\newcommand{\INDSTATE}{\STATE\quad\quad} % STATE with indention.

\usepackage{easyReview}
\setreviewsoff

%% file: SIAM_article.bbl
\begin{thebibliography}{10}

\bibitem{Aarnes2004}
{\sc J.~r.~E. Aarnes}, {\em On the use of a mixed multiscale finite element
  method for greater flexibility and increased speed or improved accuracy in
  reservoir simulation}, Multiscale Modeling \& Simulation. A SIAM
  Interdisciplinary Journal, 2 (2004), pp.~421--439,
  \url{https://doi.org/10.1137/030600655}.

\bibitem{Amestoy2019}
{\sc P.~R. Amestoy, A.~Buttari, J.-Y. L'Excellent, and T.~Mary}, {\em
  Performance and scalability of the block low-rank multifrontal factorization
  on multicore architectures}, ACM Transactions on Mathematical Software, 45
  (2019), \url{https://doi.org/10.1145/3242094}.

\bibitem{Amestoy2001}
{\sc P.~R. Amestoy, I.~S. Duff, J.-Y. L{\textquotesingle}Excellent, and
  J.~Koster}, {\em A fully asynchronous multifrontal solver using distributed
  dynamic scheduling}, {SIAM} Journal on Matrix Analysis and Applications, 23
  (2001), pp.~15--41, \url{https://doi.org/10.1137/s0895479899358194}.

\bibitem{Arbogast2007}
{\sc T.~Arbogast, G.~Pencheva, M.~F. Wheeler, and I.~Yotov}, {\em A multiscale
  mortar mixed finite element method}, Multiscale Modeling \& Simulation. A
  SIAM Interdisciplinary Journal, 6 (2007), pp.~319--346,
  \url{https://doi.org/10.1137/060662587}.

\bibitem{Arbogast1997}
{\sc T.~Arbogast, M.~F. Wheeler, and I.~Yotov}, {\em Mixed finite elements for
  elliptic problems with tensor coefficients as cell-centered finite
  differences}, SIAM Journal on Numerical Analysis, 34 (1997), pp.~828--852,
  \url{https://doi.org/10.1137/S0036142994262585}.

\bibitem{Arbogast2013}
{\sc T.~Arbogast and H.~Xiao}, {\em A multiscale mortar mixed space based on
  homogenization for heterogeneous elliptic problems}, SIAM Journal on
  Numerical Analysis, 51 (2013), pp.~377--399,
  \url{https://doi.org/10.1137/120874928}.

\bibitem{Auricchio2017}
{\sc F.~Auricchio, L.~B.~a. da~Veiga, F.~Brezzi, and C.~Lovadina}, {\em Mixed
  finite element methods}, John Wiley \& Sons, Ltd, 2017, pp.~1--53,
  \url{https://doi.org/https://doi.org/10.1002/9781119176817.ecm2004}.

\bibitem{Balay2022a}
{\sc S.~Balay, S.~Abhyankar, M.~F. Adams, S.~Benson, J.~Brown, P.~Brune,
  K.~Buschelman, E.~Constantinescu, L.~Dalcin, A.~Dener, V.~Eijkhout, W.~D.
  Gropp, V.~Hapla, T.~Isaac, P.~Jolivet, D.~Karpeev, D.~Kaushik, M.~G. Knepley,
  F.~Kong, S.~Kruger, D.~A. May, L.~C. McInnes, R.~T. Mills, L.~Mitchell,
  T.~Munson, J.~E. Roman, K.~Rupp, P.~Sanan, J.~Sarich, B.~F. Smith,
  S.~Zampini, H.~Zhang, H.~Zhang, and J.~Zhang}, {\em {PETSc/TAO} users
  manual}, Tech. Report ANL-21/39 - Revision 3.17, Argonne National Laboratory,
  2022.

\bibitem{Boffi2013}
{\sc D.~Boffi, F.~Brezzi, and M.~Fortin}, {\em Mixed finite element methods and
  applications}, vol.~44 of Springer Series in Computational Mathematics,
  Springer, Heidelberg, 2013, \url{https://doi.org/10.1007/978-3-642-36519-5}.

\bibitem{Briggs2000}
{\sc W.~L. Briggs, V.~E. Henson, and S.~F. McCormick}, {\em A multigrid
  tutorial}, Society for Industrial and Applied Mathematics (SIAM),
  Philadelphia, PA, second~ed., 2000,
  \url{https://doi.org/10.1137/1.9780898719505}.

\bibitem{Brown1997}
{\sc P.~N. Brown and H.~F. Walker}, {\em G{MRES} on (nearly) singular systems},
  SIAM Journal on Matrix Analysis and Applications, 18 (1997), pp.~37--51,
  \url{https://doi.org/10.1137/S0895479894262339}.

\bibitem{Calvo2016}
{\sc J.~G. Calvo and O.~B. Widlund}, {\em An adaptive choice of primal
  constraints for {BDDC} domain decomposition algorithms}, Electronic
  Transactions on Numerical Analysis, 45 (2016), pp.~524--544.

\bibitem{Chen2020}
{\sc J.~Chen, E.~T. Chung, Z.~He, and S.~Sun}, {\em Generalized multiscale
  approximation of mixed finite elements with velocity elimination for
  subsurface flow}, Journal of Computational Physics, 404 (2020), pp.~109133,
  23, \url{https://doi.org/10.1016/j.jcp.2019.109133}.

\bibitem{Chen2008}
{\sc Y.~Chen, T.~A. Davis, W.~W. Hager, and S.~Rajamanickam}, {\em Algorithm
  887: {CHOLMOD}, supernodal sparse {C}holesky factorization and
  update/downdate}, Association for Computing Machinery. Transactions on
  Mathematical Software, 35 (2008),
  \url{https://doi.org/10.1145/1391989.1391995}.

\bibitem{Chen2003}
{\sc Z.~Chen and T.~Y. Hou}, {\em A mixed multiscale finite element method for
  elliptic problems with oscillating coefficients}, Mathematics of Computation,
  72 (2003), pp.~541--576, \url{https://doi.org/10.1090/S0025-5718-02-01441-2}.

\bibitem{Christie2001}
{\sc M.~A. Christie and M.~J. Blunt}, {\em Tenth {SPE} comparative solution
  project: a comparison of upscaling techniques}, SPE Reservoir Evaluation \&
  Engineering, 4 (2001), pp.~308--317, \url{https://doi.org/10.2118/72469-PA},
  \url{https://doi.org/10.2118/72469-PA}.

\bibitem{Chung2015}
{\sc E.~T. Chung, Y.~Efendiev, and C.~S. Lee}, {\em Mixed generalized
  multiscale finite element methods and applications}, Multiscale Modeling \&
  Simulation. A SIAM Interdisciplinary Journal, 13 (2015), pp.~338--366,
  \url{https://doi.org/10.1137/140970574}.

\bibitem{Dolean2015}
{\sc V.~Dolean, P.~Jolivet, and F.~Nataf}, {\em An introduction to domain
  decomposition methods}, Society for Industrial and Applied Mathematics
  (SIAM), Philadelphia, PA, 2015,
  \url{https://doi.org/10.1137/1.9781611974065.ch1}.
\newblock Algorithms, theory, and parallel implementation.

\bibitem{Dolean2012}
{\sc V.~Dolean, F.~Nataf, R.~Scheichl, and N.~Spillane}, {\em Analysis of a
  two-level {S}chwarz method with coarse spaces based on local
  {D}irichlet-to-{N}eumann maps}, Computational Methods in Applied Mathematics,
  12 (2012), pp.~391--414, \url{https://doi.org/10.2478/cmam-2012-0027}.

\bibitem{Durlofsky1991}
{\sc L.~J. Durlofsky}, {\em Numerical calculation of equivalent grid block
  permeability tensors for heterogeneous porous media}, Water Resources
  Research, 27 (1991), pp.~699--708,
  \url{https://doi.org/https://doi.org/10.1029/91WR00107}.

\bibitem{Efendiev2013}
{\sc Y.~Efendiev, J.~Galvis, and T.~Y. Hou}, {\em Generalized multiscale finite
  element methods ({GM}s{FEM})}, Journal of Computational Physics, 251 (2013),
  pp.~116--135, \url{https://doi.org/10.1016/j.jcp.2013.04.045}.

\bibitem{Efendiev2012}
{\sc Y.~Efendiev, J.~Galvis, R.~Lazarov, and J.~Willems}, {\em Robust domain
  decomposition preconditioners for abstract symmetric positive definite
  bilinear forms}, ESAIM. Mathematical Modelling and Numerical Analysis, 46
  (2012), pp.~1175--1199, \url{https://doi.org/10.1051/m2an/2011073}.

\bibitem{Galvis2010}
{\sc J.~Galvis and Y.~Efendiev}, {\em Domain decomposition preconditioners for
  multiscale flows in high-contrast media}, Multiscale Modeling \& Simulation.
  A SIAM Interdisciplinary Journal, 8 (2010), pp.~1461--1483,
  \url{https://doi.org/10.1137/090751190}.

\bibitem{Galvis2010a}
{\sc J.~Galvis and Y.~Efendiev}, {\em Domain decomposition preconditioners for
  multiscale flows in high contrast media: reduced dimension coarse spaces},
  Multiscale Modeling \& Simulation. A SIAM Interdisciplinary Journal, 8
  (2010), pp.~1621--1644, \url{https://doi.org/10.1137/100790112}.

\bibitem{Graham2007}
{\sc I.~G. Graham and R.~Scheichl}, {\em Robust domain decomposition algorithms
  for multiscale {PDE}s}, Numerical Methods for Partial Differential Equations.
  An International Journal, 23 (2007), pp.~859--878,
  \url{https://doi.org/10.1002/num.20254}.

\bibitem{Griebel1995}
{\sc M.~Griebel and P.~Oswald}, {\em On the abstract theory of additive and
  multiplicative {S}chwarz algorithms}, Numerische Mathematik, 70 (1995),
  pp.~163--180, \url{https://doi.org/10.1007/s002110050115}.

\bibitem{Hajibeygi2009}
{\sc H.~Hajibeygi and P.~Jenny}, {\em Multiscale finite-volume method for
  parabolic problems arising from compressible multiphase flow in porous
  media}, Journal of Computational Physics, 228 (2009), pp.~5129--5147,
  \url{https://doi.org/10.1016/j.jcp.2009.04.017}.

\bibitem{He2021}
{\sc Z.~He, H.~Chen, J.~Chen, and Z.~Chen}, {\em Generalized multiscale
  approximation of a mixed finite element method with velocity elimination for
  {D}arcy flow in fractured porous media}, Computer Methods in Applied
  Mechanics and Engineering, 381 (2021), p.~113846,
  \url{https://doi.org/10.1016/j.cma.2021.113846}.

\bibitem{He2021a}
{\sc Z.~He, E.~T. Chung, J.~Chen, and Z.~Chen}, {\em Adaptive generalized
  multiscale approximation of a mixed finite element method with velocity
  elimination}, Computational Geosciences. Modeling, Simulation and Data
  Analysis, 25 (2021), pp.~1681--1708,
  \url{https://doi.org/10.1007/s10596-021-10068-9}.

\bibitem{Heinlein2018}
{\sc A.~Heinlein, A.~Klawonn, J.~Knepper, and O.~Rheinbach}, {\em Multiscale
  coarse spaces for overlapping {S}chwarz methods based on the {ACMS} space in
  2{D}}, Electronic Transactions on Numerical Analysis, 48 (2018),
  pp.~156--182, \url{https://doi.org/10.1553/etna\_{v}{o}{l}48s156}.

\bibitem{Heinlein2019}
{\sc A.~Heinlein, A.~Klawonn, J.~Knepper, and O.~Rheinbach}, {\em Adaptive
  {GDSW} coarse spaces for overlapping {S}chwarz methods in three dimensions},
  SIAM Journal on Scientific Computing, 41 (2019), pp.~A3045--A3072,
  \url{https://doi.org/10.1137/18M1220613}.

\bibitem{Hernandez2005}
{\sc V.~Hernandez, J.~E. Roman, and V.~Vidal}, {\em {SLEPc}}, {ACM}
  Transactions on Mathematical Software, 31 (2005), pp.~351--362,
  \url{https://doi.org/10.1145/1089014.1089019}.

\bibitem{Hou1997}
{\sc T.~Y. Hou and X.-H. Wu}, {\em A multiscale finite element method for
  elliptic problems in composite materials and porous media}, Journal of
  Computational Physics, 134 (1997), pp.~169--189,
  \url{https://doi.org/10.1006/jcph.1997.5682}.

\bibitem{Kim2017}
{\sc H.~H. Kim, E.~Chung, and J.~Wang}, {\em B{DDC} and {FETI}-{DP}
  preconditioners with adaptive coarse spaces for three-dimensional elliptic
  problems with oscillatory and high contrast coefficients}, Journal of
  Computational Physics, 349 (2017), pp.~191--214,
  \url{https://doi.org/10.1016/j.jcp.2017.08.003}.

\bibitem{Kim2018}
{\sc H.~H. Kim, E.~Chung, and J.~Wang}, {\em {BDDC} and {FETI}-{DP} algorithms
  with a change of basis formulation on adaptive primal constraints},
  Electronic Transactions on Numerical Analysis, 49 (2018), pp.~64--80,
  \url{https://doi.org/10.1553/etna_vol49s64}.

\bibitem{Klawonn2015}
{\sc A.~Klawonn, P.~Radtke, and O.~Rheinbach}, {\em F{ETI}-{DP} methods with an
  adaptive coarse space}, SIAM Journal on Numerical Analysis, 53 (2015),
  pp.~297--320, \url{https://doi.org/10.1137/130939675}.

\bibitem{Klawonn2016}
{\sc A.~Klawonn, P.~Radtke, and O.~Rheinbach}, {\em A comparison of adaptive
  coarse spaces for iterative substructuring in two dimensions}, Electronic
  Transactions on Numerical Analysis, 45 (2016), pp.~75--106.

\bibitem{Li2020}
{\sc R.~Li, H.~Yang, and C.~Yang}, {\em Parallel multilevel restricted
  {S}chwarz preconditioners for implicit simulation of subsurface flows with
  {P}eng-{R}obinson equation of state}, Journal of Computational Physics, 422
  (2020), p.~109745, \url{https://doi.org/10.1016/j.jcp.2020.109745}.

\bibitem{Lunati2006}
{\sc I.~Lunati and P.~Jenny}, {\em Multiscale finite-volume method for
  compressible multiphase flow in porous media}, Journal of Computational
  Physics, 216 (2006), pp.~616--636,
  \url{https://doi.org/10.1016/j.jcp.2006.01.001}.

\bibitem{Lunati2007}
{\sc I.~Lunati and P.~Jenny}, {\em Treating highly anisotropic subsurface flow
  with the multiscale finite-volume method}, Multiscale Modeling \& Simulation.
  A SIAM Interdisciplinary Journal, 6 (2007), pp.~308--318,
  \url{https://doi.org/10.1137/050638928}.

\bibitem{Mandel2007}
{\sc J.~Mandel and B.~Soused\'{\i}k}, {\em Adaptive selection of face coarse
  degrees of freedom in the {BDDC} and the {FETI}-{DP} iterative substructuring
  methods}, Computer Methods in Applied Mechanics and Engineering, 196 (2007),
  pp.~1389--1399, \url{https://doi.org/10.1016/j.cma.2006.03.010}.

\bibitem{Manea2016}
{\sc A.~M. Manea, J.~Sewall, and H.~A. Tchelepi}, {\em Parallel multiscale
  linear solver for highly detailed reservoir models}, SPE Journal, 21 (2016),
  pp.~2062--2078, \url{https://doi.org/10.2118/173259-PA},
  \url{https://doi.org/10.2118/173259-PA}.

\bibitem{Nataf2010}
{\sc F.~Nataf, H.~Xiang, and V.~Dolean}, {\em A two level domain decomposition
  preconditioner based on local {D}irichlet-to-{N}eumann maps}, Comptes Rendus
  Mathematique, 348 (2010), pp.~1163--1167,
  \url{https://doi.org/10.1016/j.crma.2010.10.007}.

\bibitem{Nataf2011}
{\sc F.~Nataf, H.~Xiang, V.~Dolean, and N.~Spillane}, {\em A coarse space
  construction based on local {D}irichlet-to-{N}eumann maps}, SIAM Journal on
  Scientific Computing, 33 (2011), pp.~1623--1642,
  \url{https://doi.org/10.1137/100796376}.

\bibitem{Saad2003}
{\sc Y.~Saad}, {\em Iterative methods for sparse linear systems}, Society for
  Industrial and Applied Mathematics, Philadelphia, PA, second~ed., 2003,
  \url{https://doi.org/10.1137/1.9780898718003}.

\bibitem{Sarkis1997}
{\sc M.~Sarkis}, {\em Nonstandard coarse spaces and {S}chwarz methods for
  elliptic problems with discontinuous coefficients using non-conforming
  elements}, Numerische Mathematik, 77 (1997), pp.~383--406,
  \url{https://doi.org/10.1007/s002110050292}.

\bibitem{Toselli2005}
{\sc A.~Toselli and O.~Widlund}, {\em Domain decomposition methods---algorithms
  and theory}, vol.~34 of Springer Series in Computational Mathematics,
  Springer-Verlag, Berlin, 2005, \url{https://doi.org/10.1007/b137868}.

\bibitem{wang2011}
{\sc S.~Wang, M.~V. de~Hoop, and J.~Xia}, {\em On 3{D} modeling of seismic wave
  propagation via a structured parallel multifrontal direct {H}elmholtz
  solver}, Geophysical Prospecting, 59 (2011), pp.~857--873,
  \url{https://doi.org/10.1111/j.1365-2478.2011.00982.x}.

\bibitem{Wang2014}
{\sc Y.~Wang, H.~Hajibeygi, and H.~A. Tchelepi}, {\em Algebraic multiscale
  solver for flow in heterogeneous porous media}, Journal of Computational
  Physics, 259 (2014), pp.~284--303,
  \url{https://doi.org/10.1016/j.jcp.2013.11.024}.

\bibitem{Wu2002}
{\sc X.~H. Wu, Y.~Efendiev, and T.~Y. Hou}, {\em Analysis of upscaling absolute
  permeability}, Discrete and Continuous Dynamical Systems. Series B. A Journal
  Bridging Mathematics and Sciences, 2 (2002), pp.~185--204,
  \url{https://doi.org/10.3934/dcdsb.2002.2.185}.

\bibitem{Yang2019a}
{\sc H.~Yang, S.~Sun, Y.~Li, and C.~Yang}, {\em A fully implicit
  constraint-preserving simulator for the black oil model of petroleum
  reservoirs}, Journal of Computational Physics, 396 (2019), pp.~347--363,
  \url{https://doi.org/10.1016/j.jcp.2019.05.038}.

\bibitem{Yang2023}
{\sc H.~Yang, Z.~Zhu, and J.~Kou}, {\em A minimum-type nonlinear
  complementarity simulator with constrained pressure residual ({CPR}) methods
  for wormhole propagation in carbonate acidization}, Journal of Computational
  Physics, 473 (2023), p.~111732,
  \url{https://doi.org/10.1016/j.jcp.2022.111732}.

\bibitem{Yang2022}
{\sc N.~Yang, H.~Yang, and C.~Yang}, {\em Multilevel field-split
  preconditioners with domain decomposition for steady and unsteady flow
  problems}, Computer Physics Communications, 280 (2022), p.~108496,
  \url{https://doi.org/10.1016/j.cpc.2022.108496}.

\bibitem{Yang2019}
{\sc Y.~Yang, S.~Fu, and E.~T. Chung}, {\em A two-grid preconditioner with an
  adaptive coarse space for flow simulations in highly heterogeneous media},
  Journal of Computational Physics, 391 (2019), pp.~1--13,
  \url{https://doi.org/10.1016/j.jcp.2019.03.038}.

\end{thebibliography}
